%% file: main.tex
\documentclass[onefignum,onetabnum]{siamart250211}

\input{preamble}

\headers{Restricted Dynamic Geometric Complexity}{Z. Li}

\title{Restricted Dynamic Geometric Complexity:
Path-Space Reduction and M\"obius--Jacobi Response}

\author{Zavier Li\thanks{Xidian University, Xi'an, China
  (\email{zavierli888@gmail.com}).}}

\begin{document}
\maketitle

\begin{abstract}
Structured preconditioners restrict optimization to a small family of
positive metrics, but endpoint condition-number reachability does not measure
the geometric effort required to reach a useful metric.  We formulate this
effort as a path-space value problem.  Restricted dynamic geometric complexity
is the least affine-invariant length of an admissible metric path whose
endpoint reaches a Hessian-relative generalized-eigenvalue condition target.
Path elimination gives an exact min-plus semigroup and Bellman principle,
while fixed-horizon kinetic energy is exactly squared complexity divided by
twice the horizon.  Explicit Tonelli and compact-control realizations supply
existence and a conditional Hamilton--Jacobi layer.  The main response result
is global on a Hadamard state space: geodesic convexity produces a smooth
intervention-cube path branch and a uniformly coercive Jacobi form, while one
Green inverse generates the value Hessian, two-sided force-to-curvature
bounds, exact M\"obius effects, and arbitrary prescribed finite-order
responses.  For the hard condition
target, a bordered Jacobi--KKT theorem differentiates the moving projection
endpoint and multiplier on every regular active spectral stratum; its
indefinite inverse also explains why hard-target interactions need not share
the unconstrained sign.  The theory specializes to affine-invariant
positive-definite geometry.
A determinant-one two-dimensional diagonal model has an exact target
interval, a closed-form forced path, and a strictly negative-definite
interaction matrix.  A moving diagonal Hessian gives a closed-form hard-target
projection, multiplier, and pair effects of either sign, while a
coordinate-sequential three-dimensional protocol yields an exact path metric
strictly larger than the ambient projection distance.  Thus the global Green
and bordered hard-target responses are explicit laws of restricted metric-path
elimination built on Bellman composition.
\end{abstract}



\input{sections/01_introduction}

\input{sections/02_related_work}

\input{sections/03_spd_rdgc}

\input{sections/04_path_reduction}

\input{sections/05_diagonal_model}

\input{sections/06_verification}

\input{sections/07_discussion}

\bibliographystyle{siamplain}
\bibliography{references}

\appendix
\input{appendix/a_path_proofs}

\input{appendix/b_spd_diagonal_proofs}

\end{document}

%% file: preamble.tex
\usepackage{amsmath,amssymb,mathtools,mathrsfs}
\usepackage{bm}
\usepackage{booktabs}
\usepackage{enumitem}
\usepackage{graphicx}
\usepackage{microtype}
\usepackage{placeins}

\newsiamthm{assumption}{Assumption}
\newsiamremark{remark}{Remark}

\newcommand{\R}{\mathbb R}
\newcommand{\SPD}{\mathbb S_{++}}

\newcommand{\F}{\mathfrak F}
\newcommand{\Ck}{\mathcal C_K}

\newcommand{\AI}{\mathrm{AI}}
\newcommand{\Len}{\operatorname{Len}}

\newcommand{\Val}{\operatorname{Val}}

\newcommand{\diag}{\operatorname{diag}}

\newcommand{\rank}{\operatorname{rank}}
\newcommand{\tr}{\operatorname{tr}}
\newcommand{\dist}{\operatorname{dist}}

\newcommand{\dd}{\mathrm d}

\newcommand{\Jac}{\mathcal J_{\mathrm{Jac}}}
\newcommand{\Path}{\mathscr P}
\newcommand{\Hilb}{\mathscr H}
\newcommand{\fro}[1]{\left\lVert #1\right\rVert_F}

\newcommand{\norm}[1]{\left\lVert #1\right\rVert}
\newcommand{\ip}[2]{\left\langle #1,#2\right\rangle}

\crefname{assumption}{Assumption}{Assumptions}
\crefname{definition}{Definition}{Definitions}
\crefname{theorem}{Theorem}{Theorems}
\crefname{proposition}{Proposition}{Propositions}
\crefname{lemma}{Lemma}{Lemmas}
\crefname{corollary}{Corollary}{Corollaries}

%% file: sections/01_introduction.tex
\section{Introduction}

Structured preconditioners constrain optimization to a small family of
positive metrics.  This creates an endpoint question and a path question.
The endpoint question asks whether the family can reach a prescribed
condition number.  The path question asks how much intrinsic metric motion is
required to reach it from the current state.  Endpoint feasibility alone does
not measure this geometric effort: two families may reach the same target but
place it at very different affine-invariant distances from the initial
metric.

We formulate the path question as a value problem.  For a quadratic objective
with Hessian \(H\succ0\), a metric \(G\succ0\) induces the generalized
eigenproblem \(Hv=\lambda Gv\).  In relative coordinates
\[
  S=H^{-1/2}GH^{-1/2},
\]
the relevant condition ratio is
\[
  \kappa_{\rm gen}(H,G)
  :=\frac{\lambda_{\max}(G^{-1/2}HG^{-1/2})}
  {\lambda_{\min}(G^{-1/2}HG^{-1/2})}
  =\kappa(S).
\]
This is the positive generalized-eigenvalue ratio of \(Hv=\lambda Gv\),
which is invariant under inversion of the relative spectrum.  Given a metric
family \(\F\), an initial
state \(S_0\), and \(K\ge1\), restricted dynamic geometric complexity
\(D_{K,\F}(S_0;H)\) is the least affine-invariant length of an admissible
metric path that reaches
\[
  \Ck=\{S\in\SPD^d:\kappa(S)\le K\}.
\]
The value is \(+\infty\) when the target is structurally unreachable.  The
word complexity refers throughout to minimum intrinsic metric motion (or,
after fixing the horizon, its exactly equivalent action value); it carries no
iteration-count or wall-clock complexity claim.

The central operation is elimination of an entire future path.  If
\(c_{s,t}(x,y)\) denotes the least action between two states, then restriction
and concatenation yield
\[
  c_{s,t}(x,z)
  =
  \inf_y\{c_{s,r}(x,y)+c_{r,t}(y,z)\}.
\]
This min-plus semigroup is the exact global law of path elimination.  It
precedes smoothness, existence, and Hamilton--Jacobi regularity.  Tonelli
conditions supply attained paths, while a separate compact-control
realization supplies the HJB characterization.  For fixed-horizon kinetic
action, constant-speed gauge fixing gives the exact identity
\[
  \inf_\gamma\frac12\int_0^T\norm{\dot\gamma_t}^2\,\dd t
  =
  \frac{D_{K,\F}(S_0;H)^2}{2T}.
\]
Thus the smooth action used for response analysis has precisely the same
geometric minimizers as length RDGC.

Path elimination also has a differential law.  Near a unique smooth optimum,
the vertical second variation is a Jacobi operator \(\Jac\).  Eliminating
path fluctuations gives
\[
  D\gamma_\star=-\Jac^{-1}\mathcal A_{\gamma z},
  \qquad
  D^2\Val
  =
  \mathcal A_{zz}
  -\mathcal A_{z\gamma}\Jac^{-1}\mathcal A_{\gamma z}.
\]
The second term is the path-space Schur complement.  For additive
interventions it is a negative-semidefinite Gram operator.  Its mixed entries
integrate over cube faces to exact Boolean M\"obius effects.

The local formula becomes global under geometric hypotheses.  On a Hadamard
state space, smooth geodesically convex bulk and terminal potentials yield a
unique optimal path for every intervention in a neighborhood of the Boolean
cube.  Nonpositive curvature makes the full Jacobi index form uniformly
coercive.  The Green inverse then solves bulk forcing and terminal Robin
forcing simultaneously, generates pair and conditional M\"obius effects, and
recursively determines every prescribed finite interaction order under the
corresponding smoothness.  Its Gram representation also
identifies the exact nullspace and rank of value curvature and bounds that
curvature by the intervention-force norm and the Jacobi spectrum.

The hard condition target has a different local operator.  On a regular
active spectral stratum, terminal feasibility adds one multiplier and one
linearized boundary equation.  The resulting bordered Jacobi--KKT inverse
jointly differentiates the optimal path, moving projection endpoint, and
multiplier, and yields the constrained value Hessian.  This directly treats a
moving hard RDGC target while locating the precise failure boundary at
eigenvalue collisions, loss of strict complementarity, or active-stratum
changes.  A determinant-one diagonal RDGC model realizes the unconstrained
Green theory in closed form and produces a strictly negative-definite
interaction matrix.

\paragraph{Contributions.}
The paper establishes four results.

\begin{enumerate}[leftmargin=*,itemsep=0.25em]
  \item We define RDGC as an affine-invariant path distance to a
  Hessian-relative generalized spectral target and derive the
  full-SPD spectral benchmark.  Exact min-plus composition and the
  length--energy identity provide its path-elimination and fixed-horizon
  foundations.

  \item We prove a global Hadamard M\"obius--Jacobi theorem.  One uniformly
  coercive Green operator handles bulk and terminal forcing, gives exact
  force-to-curvature bounds and rank identities, integrates to pair and
  conditional effects, and recursively generates every prescribed finite
  interaction order.

  \item We prove a bordered Jacobi--KKT theorem for a genuinely hard terminal
  inequality.  On a smooth active condition-number stratum it differentiates
  the moving RDGC projection and multiplier and gives an exact constrained
  M\"obius formula whose sign can differ from the unconstrained Gram law.

  \item We give explicit Tonelli and compact-control realizations, specialize
  the response theory to affine-invariant SPD geometry, and solve a
  two-dimensional determinant-one diagonal RDGC model exactly.  Both the
  fixed-endpoint Green response and a moving hard condition target have
  closed-form paths and interactions; the latter realizes both interaction
  signs.  A three-dimensional coordinate-sequential update protocol has an
  exact \(\ell_1\)-type RDGC strictly larger than its ambient AIRM projection,
  showing where the declared path class changes the optimization value.
\end{enumerate}

The finite-dimensional companion \cite{li2026optimizationgeometrodynamics}
proves generic Schur response and affine-intervention Gram formulas for hidden
controller states.  Those
formulas are the finite-dimensional template here.  The present paper's new
response results begin
with their path-Hilbert realization: Hadamard geometry derives global
existence and cube-wide coercivity, the weak inverse becomes a strong
Green--Jacobi problem with Robin forcing, arbitrary prescribed finite orders
share that inverse, and a bordered boundary system treats the moving hard
condition target.  The companion
information-induced-geometry paper studies exact visible metric completion
\cite{li2026informationinducedgeometry}.  The present work is
self-contained at the path-space level and does not depend on the subsequent
causal optimizer calculus.

\paragraph{Organization.}
\Cref{sec:related} positions the results relative to optimal control,
parametric optimization, SPD geometry, and interaction decompositions.
\Cref{sec:spd-rdgc} defines the affine-invariant optimization target.
\Cref{sec:path-reduction} develops path elimination and global response.
\Cref{sec:diagonal-model} gives the exact diagonal realization, and
\cref{sec:verification,sec:discussion} record verification and scope.

%% file: sections/02_related_work.tex
\section{Related Work}
\label{sec:related}

\paragraph{Variational optimization and optimal control.}
Continuous-time optimization has been studied through gradient flows,
accelerated ODEs, Bregman Lagrangians, mirror descent, and natural-gradient
geometry
\cite{su2016differential,wibisono2016variational,krichene2015accelerated,
beck2003mirror,amari1998natural}.  Our controlled variable is the metric
itself, the terminal set is determined by preconditioning quality, and the
value measures the intrinsic deformation required to reach that set.

Min-plus semigroups, dynamic programming, and Hamilton--Jacobi equations are
classical in optimal control \cite{fleming2006controlled}.  Geometric control
also supplies the classical second variation, Jacobi fields, and conjugate
point analysis \cite{agrachev2004control}.  We use these ingredients as the
analytic foundation for a different reduction target: an entire metric path
ending at a preconditioning-quality set.  The new response statements are the
intervention-cube Green representation and its bordered hard-target analogue.

\paragraph{Parametric optimization and Jacobi response.}
Reduced Hessians, variable projection, solution-map sensitivity, and implicit
differentiation eliminate optimized variables through Schur complements
\cite{golub1973differentiation,bonnans2000perturbation,
dontchev2014implicit,amos2017optnet,blondel2022efficient,
franceschi2018bilevel}.  We do not claim the generic finite-dimensional Schur
formula as new.  The contribution is its coercive path-Hilbert realization,
the derivation of that coercivity from Hadamard geometry on an entire
intervention neighborhood, the joint bulk/terminal Green operator, and the
resulting finite- and arbitrary-order interaction formulas.  For an active hard
terminal inequality, the closest general framework is parametric constrained
optimization and generalized equations
\cite{bonnans2000perturbation,dontchev2014implicit}; our bordered theorem
identifies the derivative explicitly as a Jacobi--KKT boundary operator and
specializes it to the moving condition-number projection.

Jacobi fields and index forms classically describe variations of geodesics.
Here the Jacobi operator contains kinetic, curvature, bulk-potential, and
terminal Robin terms.  Nonpositive curvature contributes a nonnegative term
to the index form, allowing global control even when the convex potentials
have linear negative growth.

The finite-dimensional companion theory already contains the generic hidden
response, Schur reduced Hessian, affine-intervention negative Gram operator,
and its pair finite-difference integral
\cite{li2026optimizationgeometrodynamics}.  The local formulas in
\cref{thm:path-schur,cor:path-interaction} are their Hilbert-path lift and are
used here as infrastructure.  The additions proved in this paper are the
Hadamard hypotheses that derive global path existence, uniqueness, and
cube-wide coercivity; the strong Green--Jacobi boundary problem that unifies
bulk and Robin forcing; arbitrary prescribed finite-order path recursion; and
the bordered Jacobi--KKT response of a moving hard target.  These conclusions
are absent from the finite-dimensional companion reduction theorem.

\paragraph{Positive-definite geometry and structured preconditioning.}
The affine-invariant geometry of positive-definite matrices supplies the
distance, exponential map, and nonpositive curvature used in the SPD
specialization
\cite{bhatia2007positive,pennec2006riemannian,higham2008functions}.
Geodesic convex analysis and metric projection on Hadamard spaces are
developed in
\cite{sra2015conic,lim2004best,bacak2014convex}.
Smoothness on a fixed simple-eigenvalue stratum and loss of classical
differentiability at spectral multiplicities follow the spectral-function
theory in \cite{lewis1996derivatives,lewis2001twice,lewis2003eigenvalue}.

Classical equilibration and condition-number minimization optimize an endpoint
scaling or matrix approximation
\cite{vandersluis1969equilibration,marechal2009optimizing,
lu2011minimizing,diamond2016stochastic,tanaka2013positive,
qu2024optimal,dogan2025geodesicpreconditioning}.  These works clarify the
endpoint benchmark addressed by \cref{thm:full-spd-benchmark}.  RDGC adds the
initial metric and an admissible path geometry; its hard-target theorem
differentiates the projection path in addition to the endpoint condition
number.

Diagonal, block, and Kronecker preconditioners arise in adaptive and
second-order optimization
\cite{duchi2011adagrad,kingma2014adam,shazeer2018adafactor,
martens2015kfac,grosse2016kfc,gupta2018shampoo,anil2020scalable}.
These algorithms motivate restricted metric families, but a concrete
optimizer-to-path statement must additionally specify the optimizer state,
damping, bias correction, step map, and interpolation.  RDGC is conditional
on that realization and measures the geometric effort once it has been
declared.

\paragraph{Finite interaction decompositions.}
Boolean and factorial interactions can be represented by finite differences,
M\"obius coefficients, and higher-order attribution indices
\cite{pukelsheim2006optimal,dasgupta2015causal,dhamdhere2020shapley,
janizek2021explaining}.  The result here is a path-geometric representation:
the pair coefficient is a cube-face integral of an inverse-Jacobi Gram kernel,
conditional effects use the same kernel on subfaces, and all higher
coefficients follow from one Green recursion.

%% file: sections/03_spd_rdgc.tex
\section{Restricted Metric Paths in Affine-Invariant Geometry}
\label{sec:spd-rdgc}

We now specialize the path system to positive-definite metrics and make the
terminal optimization problem explicit.  Let
\[
  f(\theta)=\frac12\theta^\top H\theta,
  \qquad H\in\SPD^d,
\]
and let \(G\in\SPD^d\) be a positive metric.  In Hessian-relative
coordinates,
\[
  S=H^{-1/2}GH^{-1/2}.
\]
The Hessian-relative condition ratio is
\[
  \kappa_{\rm gen}(H,G)
  :=\frac{\lambda_{\max}(G^{-1/2}HG^{-1/2})}
  {\lambda_{\min}(G^{-1/2}HG^{-1/2})}
  =\kappa(S).
\]
Whenever the shorthand \(\kappa(G^{-1}H)\) appears, it denotes this ratio of
positive generalized eigenvalues, never the Euclidean singular-value
condition number of the generally nonnormal product \(G^{-1}H\).
The affine-invariant metric and distance are
\[
  \norm{Z}_{S,\AI}
  =\fro{S^{-1/2}ZS^{-1/2}},
  \qquad
  d_{\AI}(S_0,S_1)
  =\fro{\log(S_0^{-1/2}S_1S_0^{-1/2})}.
\]
The cone \(\SPD^d\) is a finite-dimensional Hadamard manifold
\cite{bhatia2007positive,pennec2006riemannian}.

For \(K\ge1\), define the condition target
\[
  \Ck=\{S\in\SPD^d:\kappa(S)\le K\}.
\]
The unrestricted distance to this target has a one-dimensional spectral
formula.

\begin{theorem}[Full-SPD spectral benchmark]
\label{thm:full-spd-benchmark}
Let \(y_1\le\cdots\le y_d\) be the ordered log-eigenvalues of
\(S_0\in\SPD^d\).  Then
\[
  D_K(S_0)
  :=d_{\AI}(S_0,\Ck)
  =
  \min_{c\in\R}
  \left(
  \sum_{i=1}^d\dist(y_i,[c,c+\log K])^2
  \right)^{1/2}.
\]
\end{theorem}

Let \(\F\subset\SPD^d\) be an admissible family and set
\[
  \mathcal S_\F(H)
  =
  \{H^{-1/2}GH^{-1/2}:G\in\F\}.
\]
A realization includes a path-connected component, a treatment of scale
freedom, an admissible absolutely continuous path class, and its induced or
quotient affine-invariant length.

\begin{definition}[Restricted dynamic geometric complexity]
\label{def:restricted-complexity}
For \(G_0\in\F\) and \(S_0=H^{-1/2}G_0H^{-1/2}\), define
\[
\begin{aligned}
  D_{K,\F}(S_0;H)
  =
  \inf\{&\Len_{\AI}(S_{[0,T]}):
  S(0)=S_0,\ S(T)\in\Ck,\\
  &S_t\in\mathcal S_\F(H),
  \ S_{[0,T]}\text{ admissible}\}.
\end{aligned}
\]
The value is \(+\infty\) if the relative target is empty or belongs to a
different admissible component.
\end{definition}

For a closed geodesically convex realization, RDGC is an ordinary Hadamard
projection distance.  This fact connects the nonsmooth length problem to the
fixed-horizon action used in the response theory.

\begin{proposition}[Static RDGC projection]
\label{prop:rdgc-target-projection}
Suppose that \(\mathcal S_\F(H)\) is nonempty, closed, and geodesically
convex and that
\[
  \mathcal T_{K,\F}(H)
  =\mathcal S_\F(H)\cap\Ck
\]
is nonempty.  Then \(S_0\) has a unique projection
\(P_{K,\F}(S_0;H)\) onto \(\mathcal T_{K,\F}(H)\), and
\[
  D_{K,\F}(S_0;H)
  =
  d_{\AI}\bigl(S_0,P_{K,\F}(S_0;H)\bigr).
\]
For quadratic kinetic action with terminal indicator of
\(\mathcal T_{K,\F}(H)\), the value at time \(t<T\) is
\[
  V(t,S)=
  \frac{d_{\AI}(S,\mathcal T_{K,\F}(H))^2}{2(T-t)}.
\]
The minimizing path is the unique constant-speed geodesic to the projection.
\end{proposition}

The projection reduction is a benchmark for unrestricted curves in a closed
geodesically convex realization.  It does not apply when the update protocol
restricts admissible velocities or excludes the ambient geodesic.  The exact
coordinate-sequential model in \cref{prop:sequential-diagonal-rdgc} gives a
structured example whose RDGC strictly exceeds this ambient benchmark.

The convex potentials required by the global response theorem arise
naturally from relative spectra.  For \(H_a\in\SPD^d\) and a convex
permutation-invariant function \(\phi_a:\R^d\to\R\), write
\[
  \Omega_a(G)
  =
  \phi_a\!\left(
  \log\lambda(G^{-1/2}H_aG^{-1/2})
  \right).
\]

\begin{proposition}[Geodesically convex spectral information]
\label{prop:spd-spectral-convexity}
Let \(\F\subset\SPD^d\) be nonempty, closed, and geodesically convex.  If
\(\rho:\R^q\to\R\) is convex and nondecreasing in every coordinate, then
\[
  \mathcal J(G)=\rho(\Omega_1(G),\ldots,\Omega_q(G))
\]
is closed and geodesically convex on \(\F\).  If every
\(\phi_a(x+c\mathbf1)=\phi_a(x)\), then
\(\mathcal J(cG)=\mathcal J(G)\).
In particular,
\[
  \phi_{\rm osc}(x)=\max_i x_i-\min_i x_i
\]
generates \(\mathcal J_H(G)=\log\kappa_{\rm gen}(H,G)\).  Smooth spectral
approximations retain geodesic convexity.
\end{proposition}

Combining this construction with the global path theorem gives the exact SPD
interaction law used below.

\begin{corollary}[SPD--AIRM M\"obius--Jacobi specialization]
\label{cor:spd-mobius-jacobi}
Let \(\mathcal N\) be \(\SPD^d\), or a structured AIRM realization that is
a complete simply connected totally geodesic manifold.  Let the bulk and
terminal potentials in \cref{thm:global-mobius-jacobi} be smooth
geodesically convex spectral information potentials satisfying its
regularity assumptions.  Then the conclusions of that theorem hold with
\[
  J_u\xi
  =-D_t^2\xi-R^{\AI}(\xi,\dot G_u)\dot G_u
  +\operatorname{Hess}^{\AI}W_u\,\xi,
\]
and
\[
  -\ip{R^{\AI}(\xi,\dot G_u)\dot G_u}{\xi}_{G_u}\ge0.
\]
In particular, the pair effect is
\[
  \boxed{
  \zeta_{\{i,j\}}
  =-\int_{[0,1]^2}
  \left\langle
  D\Delta_i,\mathfrak J_u^{-1}D\Delta_j
  \right\rangle\dd u_i\dd u_j.}
\]
Fixed smooth active strata are covered.  The nonsmooth condition-number width
requires a generalized Jacobi inclusion at eigenvalue or active-index
changes.
\end{corollary}

\begin{remark}[Structured families covered directly]
The positive diagonal cone in a fixed basis and every fixed block-diagonal
SPD cone are complete simply connected totally geodesic AIRM submanifolds;
their geodesics, logarithms, and exponentials preserve the declared block
structure.  The determinant-one diagonal family is the corresponding flat
totally geodesic scale section.  Hence the corollaries above apply directly to
these families.  Generic low-rank and Kronecker parameterizations require
their own embedded or quotient geometry and are outside this direct
totally-geodesic specialization.
\end{remark}

The hard condition target itself also has a classical response away from
spectral collisions.  It is useful here to write
\[
  \kappa_{\rm gen}(H,G)
  =\frac{\lambda_{\max}(G^{-1/2}HG^{-1/2})}
  {\lambda_{\min}(G^{-1/2}HG^{-1/2})}.
\]

\begin{corollary}[Moving hard-condition target on a smooth spectral stratum]
\label{cor:hard-condition-response}
Let \(\F\) be a complete simply connected totally geodesic AIRM
realization, fix \(G_0\in\F\), and let \(H_u\in\SPD^d\) and \(K_u>1\) vary
smoothly.  Define
\[
  c_u(G)=\log\kappa_{\rm gen}(H_u,G)-\log K_u,
  \qquad
  \mathcal T_u=\{G\in\F:c_u(G)\le0\}.
\]
Suppose that \(G_0\notin\mathcal T_{u_0}\), that the largest and smallest
generalized eigenvalues of \((H_{u_0},q_0)\) are simple at the AIRM projection
\(q_0=P_{\mathcal T_{u_0}}G_0\), and that
\(\dd(c_{u_0}|_{\F})(q_0)\ne0\).  Then, locally in \(u\), the projection
endpoint \(q_u\), its constant-speed geodesic \(\gamma_u\), and its positive
terminal multiplier are smooth.  Moreover,
\[
  \mathcal V_{\rm hard}(u)
  =\frac{d_{\AI}(G_0,\mathcal T_u)^2}{2T}
\]
satisfies \textup{(HT4)}--\textup{(HT5)} with
\(\mathfrak J_u\) equal to the kinetic index form plus
\(\lambda_u\operatorname{Hess}^{\AI}(c_u|_{\F})\) at the endpoint.
If the active defining function is affine in the intervention coordinates on
a Boolean face, and the same active spectral stratum, simple extreme
eigenvalues, LICQ, positive multiplier, and bordered invertibility persist on
that entire face, \textup{(HT6)} gives its exact hard-RDGC interaction.

Thus the response differentiates the moving hard-target projection itself.
The conclusion ends precisely when an extreme generalized eigenvalue loses
simplicity, the restricted boundary loses regularity, strict complementarity
fails, or the active stratum changes.
\end{corollary}

%% file: sections/04_path_reduction.tex
\section{Path-Space Reduction and Jacobi Response}
\label{sec:path-reduction}

This section separates the algebraic, variational, and differential layers of
restricted geometric control.  The first layer needs only path concatenation.
The second adds compactness and convexity.  The third studies a smooth
nondegenerate optimal path through its second variation.

\subsection{Transition costs and min-plus composition}

Let \(\mathcal M\) be a state space and let
\(\Path_{s,t}(x,y)\) be a declared class of admissible paths from \(x\) at
time \(s\) to \(y\) at time \(t\).  A path has action
\(\mathcal A_{s,t}(\gamma)\in\R\cup\{+\infty\}\).  Empty path classes have
value \(+\infty\).

\begin{assumption}[Composable path system]
\label{ass:path-composition}
For every \(s<r<t\):
\begin{enumerate}[leftmargin=*,itemsep=0.15em]
  \item restricting \(\gamma\in\Path_{s,t}(x,z)\) at time \(r\) produces
  paths in \(\Path_{s,r}(x,\gamma(r))\) and
  \(\Path_{r,t}(\gamma(r),z)\);
  \item any two admissible segments with a common state at time \(r\) can be
  concatenated into an admissible path;
  \item the action is additive under restriction and concatenation.
  \item for every fixed interval \([a,b]\), there is
  \(m_{a,b}>-\infty\) such that
  \(\mathcal A_{a,b}(\gamma)\ge m_{a,b}\) for every admissible segment.
\end{enumerate}
\end{assumption}

Define the transition cost
\begin{equation*}
  c_{s,t}(x,y)
  =
  \inf_{\gamma\in\Path_{s,t}(x,y)}\mathcal A_{s,t}(\gamma).
\end{equation*}

\begin{theorem}[Path-elimination semigroup]
\label{thm:path-dpp}
Under \cref{ass:path-composition}, for every \(s<r<t\),
\begin{equation}
  \boxed{
  c_{s,t}(x,z)
  =
  \inf_{y\in\mathcal M}
  \{c_{s,r}(x,y)+c_{r,t}(y,z)\}.}
  \tag{P1}
\end{equation}
Both sides lie in \(\R\cup\{+\infty\}\), so the min-plus sum is always
defined.  The equality does not require minimizers, smoothness, or convexity.
\end{theorem}

For a terminal cost \(\Gamma:\mathcal M\to\R\cup\{+\infty\}\) bounded
below, define
\[
  V(s,x)
  =
  \inf_y\{c_{s,T}(x,y)+\Gamma(y)\}.
\]

\begin{corollary}[Bellman principle]
\label{cor:path-bellman}
Under the same assumptions,
\begin{equation}
  \boxed{
  V(s,x)
  =
  \inf_y\{c_{s,r}(x,y)+V(r,y)\},
  \qquad s<r<T.}
  \tag{P2}
\end{equation}
Thus Bellman recursion is associativity of path elimination across time
slices.
\end{corollary}

\subsection{The fixed-horizon energy gauge}

Let \(\mathcal C\subset\mathcal M\) be a terminal set and let
\(\Path_T(x,\mathcal C)\) be an admissible class of absolutely continuous
paths on \([0,T]\) starting at \(x\) and ending in \(\mathcal C\).  Suppose
the class is closed under orientation-preserving absolutely continuous
reparameterization and every finite-length path admits its constant-speed
reparameterization within the class.  Define
\[
  D(x,\mathcal C)=\inf_{\gamma\in\Path_T(x,\mathcal C)}\Len(\gamma),
  \qquad
  \mathcal E_T(x,\mathcal C)
  =\inf_{\gamma\in\Path_T(x,\mathcal C)}
  \frac12\int_0^T\norm{\dot\gamma_t}^2\,\dd t.
\]

\begin{theorem}[Length--energy gauge for RDGC]
\label{thm:length-energy-gauge}
Under the preceding reparameterization assumptions,
\begin{equation}
  \boxed{
  \mathcal E_T(x,\mathcal C)
  =\frac{D(x,\mathcal C)^2}{2T}.}
  \tag{P3}
\end{equation}
The identity holds with value \(+\infty\).  When the common value is finite,
a finite-energy path attains the energy minimum if and only if it has constant
speed and attains the length minimum.  Consequently, whenever
the RDGC path class has this reparameterization property,
\[
  \mathcal E_{K,\F,T}(S_0;H)
  =\frac{D_{K,\F}(S_0;H)^2}{2T}.
\]
Thus fixed time is a gauge choice that removes the reparameterization
degeneracy of the length functional without changing its optimal geometric
path.
\end{theorem}

\begin{remark}[A checkable intervention-cube margin]
The convexity hypothesis (MJ2) follows from a uniform Hessian margin.  For
example, suppose \(\mathcal U_0\subset(-\delta,1+\delta)^m\),
\[
  \operatorname{Hess}W_0\succeq\alpha_W g,
  \qquad
  \norm{\operatorname{Hess}w_i}_{\mathrm{op}}\le b_i,
\]
uniformly on \([0,T]\times\mathcal N\), and
\(\alpha_W>(1+\delta)\sum_i b_i\).  Then every \(W_u\) is geodesically
convex on \(\mathcal U_0\).  The analogous condition
\(\alpha_\Phi>(1+\delta)\sum_i c_i\) for
\(\operatorname{Hess}\Phi_0\succeq\alpha_\Phi g\) and
\(\norm{\operatorname{Hess}\phi_i}_{\mathrm{op}}\le c_i\) treats the terminal
potential.  Affine interventions have \(b_i=c_i=0\), as in
\cref{thm:diag-forced-response}.
\end{remark}

\subsection{Existence and Hamilton--Jacobi characterization}

We next give one sufficient analytic realization.  Let
\((\mathcal M,g)\) be a finite-dimensional Hadamard manifold and
\(\F\subset\mathcal M\) a nonempty closed geodesically convex set.  Admissible
paths are absolutely continuous curves in \(\F\) satisfying
\(\dot\gamma(t)\in\mathcal A_t(\gamma(t))\) almost everywhere, and
\[
  \mathcal A_{s,t}(\gamma)
  =
  \int_s^t
  \{L_r(\gamma(r),\dot\gamma(r))+U(r,\gamma(r))\}\,\dd r.
\]

\begin{assumption}[Tonelli--viability conditions]
\label{ass:tonelli}
The admissible velocity multifunction has nonempty closed convex values and a
Borel graph.  In every smooth coordinate chart, the extended integrand
\[
  \overline L(t,G,Z)
  =L_t(G,Z)+U(t,G)+\iota_{\mathcal A_t(G)}(Z)
\]
is a normal integrand: it is Borel measurable in \(t\), lower semicontinuous
in \((G,Z)\) for almost every \(t\), and convex in \(Z\).  For some
\(p>1\), \(a>0\), and \(b\ge0\), uniformly in \((t,G,Z)\),
\[
  \overline L(t,G,Z)\ge a\norm{Z}_G^p-b.
\]
The viability constraint is sequentially closed under uniformly convergent
base paths and weakly convergent chart velocities.  The terminal cost
\(\Gamma\) is proper, lower semicontinuous, and bounded below, and at least one
finite-action competitor exists whenever the value under consideration is
finite.  The path system is composable in the sense of
\cref{ass:path-composition}.
\end{assumption}

\begin{theorem}[Tonelli existence for restricted paths]
\label{thm:path-tonelli}
Under \cref{ass:tonelli}, every finite value \(V(s,x)\) is attained.  If
pointwise geodesic interpolation preserves admissibility and the complete
action plus terminal cost is strictly geodesically convex for paths with a
common initial state, then the minimizing path is unique.
\end{theorem}

For the HJB layer we use a separate, directly checkable control realization.
Suppose \(\F\) is a smooth complete manifold without boundary,
\(\mathcal U\) is a compact metric control space, and admissible paths solve
\[
  \dot G_t=b(t,G_t,a_t)
\]
for measurable controls \(a_t\in\mathcal U\).  Assume \(b\) is continuous,
locally Lipschitz in \(G\) uniformly in \((t,a)\), and tangent to \(\F\).
Let the running cost \(\ell(t,G,a)\) and terminal cost \(\Gamma(G)\) be
continuous and bounded below, and define
\[
  \mathcal H(t,G,p)
  =\sup_{a\in\mathcal U}
  \{\ip{p}{b(t,G,a)}-\ell(t,G,a)\}.
\]

\begin{proposition}[HJB layer with explicit analytic boundary]
\label{prop:path-hjb}
Assume the preceding compact-control realization is nonexplosive on
\([0,T]\), its value \(V\) is finite and continuous, and admissible controls
are closed under concatenation.  Then \(V\) is a viscosity solution of
\begin{equation}
  -\partial_tV(t,G)
  +\mathcal H(t,G,-\dd_GV(t,G))=0,
  \qquad
  V(T,\cdot)=\Gamma.
  \tag{P4}
\end{equation}
At every differentiability point this equation holds classically.  If the
Hamiltonian and terminal data satisfy a comparison principle in a declared
solution class, then \(V\) is the unique viscosity solution in that class.
When the maximizing control is unique, the optimal path satisfies
\[
  \dot G_t=b(t,G_t,a^\star(t,G_t,-\dd_GV(t,G_t)))
\]
at classical points.  Compactness of \(\mathcal U\) gives maximizers and local
velocity compactness; continuity and the uniform local Lipschitz condition
give the short-time realizability used in the viscosity test.  No comparison
or uniqueness claim is made without the stated comparison hypothesis.
\end{proposition}

The layers above have different force.  \Cref{thm:path-dpp} is an exact
algebraic identity, \cref{thm:length-energy-gauge} fixes the time gauge without
changing the optimal geometric path, \cref{thm:path-tonelli} is a direct-method
theorem, and \cref{prop:path-hjb} is an analytic characterization for the
declared compact-control realization, conditional on comparison only for
uniqueness.

\subsection{Jacobi--Schur response}

Let \(\mathcal Z\) be a finite-dimensional visible parameter space and
\(\Hilb_0\) a real Hilbert space of zero-boundary path fluctuations.  A local
trivialization of the admissible path space writes
\[
  \gamma=Rz+\eta,
  \qquad z\in\mathcal Z,
  \quad \eta\in\Hilb_0,
\]
where \(R\) is a fixed continuous right inverse for the visible boundary or
constraint map.  Let
\(\mathfrak A:\mathcal Z\times\Hilb_0\to\R\) be the resulting action.
We use the Riesz isomorphism to identify \(\Hilb_0^*\) with \(\Hilb_0\).

\begin{theorem}[Jacobi--Schur path-response law]
\label{thm:path-schur}
Suppose \(\mathfrak A\in C^3\) near \((z_0,\eta_0)\), where
\(\eta_0\) is the unique local minimizer at \(z_0\).  Assume the vertical
second variation
\[
  \Jac
  =D^2_{\eta\eta}\mathfrak A(z_0,\eta_0)
\]
is represented by a bounded self-adjoint operator satisfying
\[
  \ip{h}{\Jac h}\ge\mu\norm{h}^2,
  \qquad \mu>0,
  \quad h\in\Hilb_0.
\]
Then there are neighborhoods \(Z_0\) of \(z_0\) and \(N_0\) of \(\eta_0\)
such that every \(z\in Z_0\) has a unique minimizing critical point
\(\eta_\star(z)\in N_0\), and this local minimizing section is \(C^2\).  Write
\[
  \Jac(z)=D^2_{\eta\eta}\mathfrak A(z,\eta_\star(z)).
\]
With
\(\Val(z)=\mathfrak A(z,\eta_\star(z))\),
\begin{equation}
  \boxed{
  D\eta_\star
  =
  -\Jac^{-1}\mathfrak A_{\eta z},}
  \tag{P5}
\end{equation}
and
\begin{equation}
  \boxed{
  D^2\Val
  =
  \mathfrak A_{zz}
  -\mathfrak A_{z\eta}\Jac^{-1}\mathfrak A_{\eta z}.}
  \tag{P6}
\end{equation}
All derivatives on the right are evaluated along the minimizing path.  The
effective value curvature is therefore the Schur complement of the vertical
Jacobi block.
\end{theorem}

\begin{corollary}[Coercive kinetic bridge on a Hadamard family]
\label{cor:rdgc-jacobi-bridge}
Let \(\F\) be a finite-dimensional Hadamard manifold realized as a totally
geodesic submanifold of the ambient geometry, and let \(\gamma_0\) be the constant-speed geodesic from
\(x\in\F\) to a fixed endpoint \(q\in\F\).  On zero-boundary fluctuations
\(h\in H_0^1(\gamma_0^*T\F)\), the second variation of kinetic energy is
\[
  I(h,h)=\int_0^T
  \left(\norm{D_th}^2
  -\ip{R(\dot\gamma_0,h)h}{\dot\gamma_0}\right)\dd t
  \ge\int_0^T\norm{D_th}^2\dd t.
\]
Hence the weak Jacobi operator is coercive in the derivative norm on
\(H_0^1\).  If \(q\) is the unique RDGC target projection, then
\cref{thm:length-energy-gauge} identifies \(\gamma_0\) with the energy-gauged
RDGC minimizer.  Every \(C^3\) endpoint perturbation that preserves the fixed
boundary chart therefore satisfies the full Schur formula in
\cref{thm:path-schur}, including its direct visible Hessian term.  Affine
additive-action intensities with fixed endpoint additionally satisfy the
negative-Gram formula in \cref{cor:path-interaction} locally.
\end{corollary}

\subsection{Global Hadamard M\"obius--Jacobi realization}

The abstract Schur law becomes a global theorem on a whole intervention cube
for mechanical actions on Hadamard path space.  Let \((\mathcal N,g)\) be a
finite-dimensional Hadamard manifold, fix \(x_0\in\mathcal N\) and \(T>0\),
and set
\[
  \Omega_{x_0}
  =\{\gamma\in H^1([0,T],\mathcal N):\gamma(0)=x_0\}.
\]
Let the open set \(\mathcal U_0\Subset\mathcal U\subset\R^m\) contain
\([0,1]^m\), and
consider
\[
  W_u=W_0+\sum_{i=1}^m u_iw_i,
  \qquad
  \Phi_u=\Phi_0+\sum_{i=1}^m u_i\phi_i,
\]
\begin{equation}
  \mathcal A_u(\gamma)
  =\int_0^T
  \left(\frac12\norm{\dot\gamma}^2+W_u(t,\gamma(t))\right)\dd t
  +\Phi_u(\gamma(T)).
  \tag{MJ1}
\end{equation}

\begin{theorem}[Global path-space M\"obius--Jacobi theorem]
\label{thm:global-mobius-jacobi}
Fix \(r\ge2\).  Assume the bulk potentials are continuous in \(t\), all bulk
and terminal potentials are \(C^{r+2}\) in the state variable with covariant
derivatives locally uniform in \((t,u)\), and, for every
\(u\in\mathcal U_0\),
\begin{equation}
  \operatorname{Hess}W_u(t,\cdot)\succeq0,
  \qquad
  \operatorname{Hess}\Phi_u\succeq0.
  \tag{MJ2}
\end{equation}
Assume also that the values and gradients of \(W_u(t,\cdot)\) and \(\Phi_u\)
at \(x_0\) are uniformly bounded over \([0,T]\times\mathcal U_0\).
Then the following conclusions hold.

\begin{enumerate}[leftmargin=*,itemsep=0.35em]
  \item For every \(u\in\mathcal U_0\), \(\mathcal A_u\) has a unique global
  minimizer \(\gamma_u\), and \(u\mapsto\gamma_u\) is \(C^r\) on
  \(\mathcal U_0\).  No global lower bound on the potentials is required.
  The minimizer satisfies
  \begin{equation}
    D_t\dot\gamma_u=\operatorname{grad}W_u(t,\gamma_u),
    \qquad
    \dot\gamma_u(T)+\operatorname{grad}\Phi_u(\gamma_u(T))=0.
    \tag{MJ3}
  \end{equation}

  \item On
  \(\mathcal V_u=\{\xi\in H^1(\gamma_u^*T\mathcal N):\xi(0)=0\}\), the
  second variation is
  \begin{equation}
  \begin{aligned}
    \mathfrak B_u(\xi,\eta)
    ={}&\int_0^T\bigl[
    \ip{D_t\xi}{D_t\eta}
    -\ip{R(\xi,\dot\gamma_u)\dot\gamma_u}{\eta}
    +\operatorname{Hess}W_u[\xi,\eta]\bigr]\dd t\\
    &+\operatorname{Hess}\Phi_u[\xi(T),\eta(T)],
  \end{aligned}
  \tag{MJ4}
  \end{equation}
  and it obeys the uniform bound
  \begin{equation}
    \mathfrak B_u(\xi,\xi)
    \ge\int_0^T\norm{D_t\xi}^2\dd t.
    \tag{MJ5}
  \end{equation}
  Hence its weak Jacobi map
  \(\mathfrak J_u:\mathcal V_u\to\mathcal V_u^*\) is a positive
  self-adjoint isomorphism.

  \item Define
  \[
    \Delta_i(\gamma)
    =\int_0^T w_i(t,\gamma(t))\dd t+\phi_i(\gamma(T)),
    \qquad
    \mathfrak f_i^u=D\Delta_i(\gamma_u),
  \]
  and \(\chi_i^u=\mathfrak J_u^{-1}\mathfrak f_i^u\).  Then
  \begin{equation}
    \partial_{u_i}\gamma_u=-\chi_i^u.
    \tag{MJ6}
  \end{equation}
  The response is equivalently the unique solution of
  \begin{equation}
  \begin{aligned}
    J_u\chi_i^u&=\operatorname{grad}w_i(t,\gamma_u),
    &\chi_i^u(0)&=0,\\
    D_t\chi_i^u(T)+\operatorname{Hess}\Phi_u\,\chi_i^u(T)
    &=\operatorname{grad}\phi_i(\gamma_u(T)),
  \end{aligned}
  \tag{MJ7}
  \end{equation}
  where
  \[
    J_u\xi=-D_t^2\xi-R(\xi,\dot\gamma_u)\dot\gamma_u
    +\operatorname{Hess}W_u\,\xi.
  \]
  Thus bulk and terminal intervention forces share one Green--Jacobi inverse.

  \item The value \(\mathcal V(u)=\mathcal A_u(\gamma_u)\) is \(C^r\) and
  \begin{equation}
    \partial_{u_i}\mathcal V=\Delta_i(\gamma_u),
    \qquad
    \partial_{u_i u_j}^2\mathcal V
    =-\mathfrak B_u(\chi_i^u,\chi_j^u).
    \tag{MJ8}
  \end{equation}
  Equivalently,
  \(D_{uu}^2\mathcal V=-\mathfrak F_u^*\mathfrak J_u^{-1}\mathfrak F_u
  \preceq0\), where \(\mathfrak F_ua=\sum_i a_i\mathfrak f_i^u\).
  More precisely,
  \begin{equation}
  \begin{aligned}
    \frac{\norm{\mathfrak F_ua}_{\mathcal V_u^*}^2}
    {\norm{\mathfrak J_u}}
    &\le -a^\top D_{uu}^2\mathcal V(u)a\\
    &\le
    \norm{\mathfrak J_u^{-1}}
    \norm{\mathfrak F_ua}_{\mathcal V_u^*}^2,
  \end{aligned}
  \tag{MJ8a}
  \end{equation}
  and therefore
  \begin{equation}
    \ker(-D_{uu}^2\mathcal V)=\ker\mathfrak F_u,
    \qquad
    \rank(-D_{uu}^2\mathcal V)=\rank\mathfrak F_u.
    \tag{MJ8b}
  \end{equation}

  \item For \(F(A)=\mathcal V(\mathbf1_A)\) and
  \(\zeta_S=\sum_{B\subseteq S}(-1)^{|S|-|B|}F(B)\), every
  \(|S|\le r\) satisfies
  \begin{equation}
    \zeta_S
    =\int_{[0,1]^S}
    \partial_S^{|S|}\mathcal V
    \left(\sum_{i\in S}u_ie_i\right)\prod_{i\in S}\dd u_i.
    \tag{MJ9}
  \end{equation}
  In particular,
  \begin{equation}
  \boxed{
    \zeta_{\{i,j\}}
    =-\int_0^1\!\int_0^1
    \mathfrak B_{se_i+te_j}
    (\chi_i^{se_i+te_j},\chi_j^{se_i+te_j})\,\dd s\,\dd t.}
    \tag{MJ10}
  \end{equation}
  The same formula on the face
  \(u=\mathbf1_C+se_i+te_j\) gives the conditional pair contrast for every
  background \(C\cap\{i,j\}=\varnothing\), with every coordinate outside
  \(C\cup\{i,j\}\) fixed at zero.
\end{enumerate}
\end{theorem}

\begin{corollary}[Arbitrary finite-order Green--Jacobi recursion]
\label{cor:all-order-mobius-jacobi}
Choose a smooth local path chart, let \(I\subseteq[m]\) satisfy
\(1\le|I|\le r\), and write
\(\eta_I^u=\partial_I^{|I|}\gamma_u\), and interpret the displayed higher
derivatives in that chart (equivalently, under one fixed covariant
differentiation convention).  Let \(\Pi(J)\) denote the set of
partitions of \(J\), with \(\Pi(\varnothing)=\{\varnothing\}\), and put
\(\mathscr F(u,\gamma)=D_\gamma\mathcal A_u(\gamma)\).  Then
\begin{equation}
  \eta_I^u=-\mathfrak J_u^{-1}\mathcal R_I^u,
  \tag{MJ11}
\end{equation}
where
\begin{equation}
  \mathcal R_I^u
  =\sum_{\substack{S\subseteq I,\ |S|\le1,\ \pi\in\Pi(I\setminus S)\\
  (S,\pi)\ne(\varnothing,\{I\})}}
  D_u^S D_\gamma^{|\pi|}\mathscr F
  \bigl[\eta_B^u\bigr]_{B\in\pi}.
  \tag{MJ12}
\end{equation}
Thus the same inverse Jacobi operator recursively generates every path
derivative through order \(r\) and, through (MJ9), every Boolean M\"obius
effect of that order.  Since \(r\) may be prescribed arbitrarily, the result
covers any fixed finite order under the corresponding \(C^{r+2}\) regularity;
\(C^\infty\) data give all finite orders simultaneously.  If \(r\ge3\), then
for distinct \(i,j,k\),
\begin{equation}
  \boxed{
  \begin{aligned}
  \partial_{u_i u_j u_k}^3\mathcal V
  ={}&-D_\gamma^3\mathcal A_u[\chi_i,\chi_j,\chi_k]\\
  &+D_\gamma^2\Delta_i[\chi_j,\chi_k]
  +D_\gamma^2\Delta_j[\chi_i,\chi_k]
  +D_\gamma^2\Delta_k[\chi_i,\chi_j].
  \end{aligned}}
  \tag{MJ13}
\end{equation}
\end{corollary}

The intermediate chart derivatives \(\eta_I^u\) and the displayed recursive
representation depend on the chosen path chart, or equivalently on the fixed
covariant-differentiation convention.  The scalar derivatives
\(\partial_I^{|I|}\mathcal V\) and the resulting Boolean M\"obius effects are
intrinsic.

\subsection{Hard-target response on a smooth active stratum}

We now differentiate the terminal constraint itself.  This gives the missing
hard-target counterpart of the Robin response above.  Let
\(c:\mathcal U\times\mathcal N\to\R\) be smooth and consider
\begin{equation}
  \mathcal V_{\rm hard}(u)
  =\inf\{\mathcal A_u(\gamma):
  \gamma\in\Omega_{x_0},\ c_u(\gamma(T))\le0\}.
  \tag{HT1}
\end{equation}
The statement is local in the parameter and in one smooth active stratum;
the inequality itself remains hard.

\begin{theorem}[Bordered Jacobi response for an active hard target]
\label{thm:hard-target-jacobi}
Fix \(u_0\in\mathcal U\).  Suppose that \(\mathcal A_u\) and \(c_u\) are
\(C^{r+2}\), \(r\ge2\), in compatible local path and state charts.  Let
\(\gamma_0\) be a constrained local minimizer of \textup{(HT1)} with active
constraint, and suppose that there is a strictly positive multiplier
\(\lambda_0\) such that, for
\[
  \mathscr L(u,\gamma,\lambda)
  =\mathcal A_u(\gamma)+\lambda c_u(\gamma(T)),
\]
the KKT equations are
\begin{equation}
  D_\gamma\mathscr L(u_0,\gamma_0,\lambda_0)=0,
  \qquad c_{u_0}(\gamma_0(T))=0.
  \tag{HT2}
\end{equation}
On
\(\mathcal V_0=\{\xi\in H^1(\gamma_0^*T\mathcal N):\xi(0)=0\}\), set
\[
  \mathfrak J_0=D_{\gamma\gamma}^2\mathscr L,
  \qquad
  \mathfrak B_0\xi
  =\dd c_{u_0}(\gamma_0(T))[\xi(T)].
\]
Assume that \(\mathfrak J_0:\mathcal V_0\to\mathcal V_0^*\) is bounded,
self-adjoint, and coercive, and that \(\mathfrak B_0\ne0\).  Then:
\begin{enumerate}[leftmargin=*,itemsep=0.3em]
  \item there is a neighborhood \(U_0\) of \(u_0\) and a unique \(C^r\)
  branch \(u\mapsto(\gamma_u,\lambda_u)\) of active KKT points near
  \((\gamma_0,\lambda_0)\), with \(\lambda_u>0\); every \(\gamma_u\) is the
  unique constrained local minimizer in a common path neighborhood;

  \item on \(\mathcal Y_u=\mathcal V_u\times\R\), the bordered Jacobi--KKT
  map
  \begin{equation}
    \mathfrak K_u
    =
    \begin{pmatrix}
      \mathfrak J_u&\mathfrak B_u^*\\
      \mathfrak B_u&0
    \end{pmatrix}
    :\mathcal Y_u\longrightarrow\mathcal Y_u^*
    \tag{HT3}
  \end{equation}
  is a self-adjoint isomorphism, where
  \(\mathfrak J_u=D_{\gamma\gamma}^2\mathscr L\) and
  \(\mathfrak B_u\xi=\dd c_u(\gamma_u(T))[\xi(T)]\);

  \item if
  \[
    \mathfrak b_i^u
    =D_{(\gamma,\lambda)}\partial_{u_i}\mathscr L
      (u,\gamma_u,\lambda_u)\in\mathcal Y_u^*,
  \]
  then the path and multiplier response and the constrained value Hessian are
  \begin{equation}
    \boxed{
    \partial_{u_i}(\gamma_u,\lambda_u)
    =-\mathfrak K_u^{-1}\mathfrak b_i^u,}
    \tag{HT4}
  \end{equation}
  \begin{equation}
    \boxed{
    \partial_{u_i u_j}^2\mathcal V_{\rm hard}
    =\partial_{u_i u_j}^2\mathscr L
    -\left\langle
      \mathfrak b_i^u,\mathfrak K_u^{-1}\mathfrak b_j^u
    \right\rangle_{\mathcal Y_u^*,\mathcal Y_u}.}
    \tag{HT5}
  \end{equation}
  The direct derivative in \textup{(HT5)} holds \((\gamma,\lambda)\) fixed.
\end{enumerate}
If \(\mathcal A_u\) and \(c_u\) are affine in \(u\), the direct term in
\textup{(HT5)} vanishes.  If the same active branch and bordered
invertibility persist on a Boolean two-face, its exact conditional pair
effect is
\begin{equation}
  \boxed{
  \zeta_{\{i,j\}\mid C}
  =-\int_0^1\!\int_0^1
  \left\langle
  \mathfrak b_i^u,\mathfrak K_u^{-1}\mathfrak b_j^u
  \right\rangle\,\dd s\,\dd t,\quad
  u=\mathbf1_C+se_i+te_j.}
  \tag{HT6}
\end{equation}
Here every coordinate outside \(C\cup\{i,j\}\) is fixed at zero.
Unlike the unconstrained Gram formula, the bordered inverse is indefinite, so
hard-target interactions need not have one sign.
\end{theorem}

The terminal contribution to \(\mathfrak J_u\) is
\(\lambda_u\operatorname{Hess}c_u\), while
\(\mathfrak B_u\) linearizes motion of the active boundary.  Thus
\textup{(HT3)} couples the interior Jacobi equation, endpoint transversality,
and active-target motion in one invertible system.

For fixed endpoints, the same theorem holds on Dirichlet fluctuations and the
terminal forces disappear.  For a general regular Lagrangian it remains valid
when existence, uniqueness, smoothness, and uniform coercivity of the complete
second variation are verified.  At active-set changes the correct object is a
normal-cone graphical derivative coupled to a Jacobi variational inequality;
a classical Hessian need not exist.

Accordingly, \cref{thm:global-mobius-jacobi} is the smooth mechanical
extension of the RDGC path framework, while
\cref{thm:hard-target-jacobi,cor:hard-condition-response} gives the classical
hard-target response on a fixed smooth active stratum.  Neither theorem
asserts classical differentiability through eigenvalue multiplicities or
active-index changes.

The theorem applies to a smooth integral action when its second-variation
bilinear form is continuous and coercive on the chosen Sobolev fluctuation
space.  The Riesz map identifies that bilinear form with \(\Jac\); in a
classical Euler--Lagrange problem this is the weak Jacobi operator with the
declared boundary conditions.

\begin{corollary}[Associative second-order path elimination]
\label{cor:path-schur-associativity}
If \(\eta=(\eta_1,\eta_2)\) and the complete vertical Hessian is coercive,
then eliminating \(\eta_2\) and subsequently \(\eta_1\) gives the same value
Hessian as eliminating \((\eta_1,\eta_2)\) in one step:
\begin{equation}
  \operatorname{Schur}_{(\eta_1,\eta_2)}\mathbb H
  =
  \operatorname{Schur}_{\eta_1}
  \bigl(\operatorname{Schur}_{\eta_2}\mathbb H\bigr).
  \tag{P7}
\end{equation}
\end{corollary}

\subsection{Interaction curvature generated by path relaxation}

Let \(u\in U\subset\R^m\) parameterize smooth mechanism intensities and
suppose the unreduced action is affine in \(u\):
\[
  \mathfrak A(\eta,u)
  =
  \mathfrak A_0(\eta)+\sum_{i=1}^m u_i\Delta_i(\eta).
\]
Define \(\mathcal G:\R^m\to\Hilb_0^*\) by
\[
  \mathcal G a=\sum_i a_iD_\eta\Delta_i.
\]

The finite-dimensional form of the following Schur identity is established
in the companion hidden-state reduction theory
\cite{li2026optimizationgeometrodynamics}.  We record its path-Hilbert form to
connect the abstract reduction formula to the Jacobi and Green operators
constructed above; the generic negative-Gram algebra is used as an established
ingredient.

\begin{corollary}[Path interaction curvature]
\label{cor:path-interaction}
Under \cref{thm:path-schur}, the local reduced value satisfies
\begin{equation}
  \boxed{
  D^2_{uu}\Val
  =
  -\mathcal G^*\Jac^{-1}\mathcal G
  \preceq0.}
  \tag{P8}
\end{equation}
In components,
\begin{equation}
  \partial^2_{u_i u_j}\Val
  =
  -\ip{D_\eta\Delta_i}
  {\Jac^{-1}D_\eta\Delta_j}.
  \tag{P9}
\end{equation}
If the same minimizing branch exists smoothly and the vertical second
variation remains uniformly coercive on an open set containing
\([0,1]^m\), then for two Boolean endpoint settings, with all other
intensities fixed at zero, the mixed finite difference is exactly
\begin{equation}
\begin{aligned}
  &\Val(e_i+e_j)-\Val(e_i)-\Val(e_j)+\Val(0)\\
  &\quad=
  -\int_0^1\!\int_0^1
  \ip{D_\eta\Delta_i}
  {\Jac^{-1}D_\eta\Delta_j}_{\eta_\star(se_i+te_j)}
  \,\dd s\,\dd t.
\end{aligned}
\tag{P10}
\end{equation}
Thus hidden path relaxation creates an observable interaction whenever the two
path perturbations fail to be orthogonal in the inverse-Jacobi metric.  A zero
integrated interaction can also arise through cancellation; only pointwise
zero curvature implies pointwise orthogonality.
\end{corollary}

%% file: sections/05_diagonal_model.tex
\section{A Closed-Form Diagonal RDGC Model}
\label{sec:diagonal-model}

The general theory has a nontrivial realization in the determinant-one
two-dimensional diagonal family.  Write
\[
  D(u)=\diag(e^{u/2},e^{-u/2}),
\]
so the global scale has been quotiented out.  The induced affine-invariant
length is
\[
  \Len_{\diag}(u_{[0,T]})
  =
  \frac{1}{\sqrt2}\int_0^T|\dot u_t|\,\dd t.
\]
For \(H\in\SPD^2\), define
\[
  K_{\diag}^\ast(H)
  =\inf_{u\in\R}\kappa_{\rm gen}(H,D(u)).
\]

\begin{theorem}[Two-dimensional diagonal exact complexity]
\label{thm:2d-diag}
Let \(H\in\SPD^2\), \(K\ge1\), and
\[
  R_K=\frac{K+1}{\sqrt K}\sqrt{\det H}.
\]
If \(K<K_{\diag}^\ast(H)\), then
\[
  D_{K,\diag}^{2D}(u_0;H)=+\infty.
\]
If \(K\ge K_{\diag}^\ast(H)\), define
\[
  x_\pm(K)
  =
  \frac{R_K\pm\sqrt{R_K^2-4H_{11}H_{22}}}{2H_{22}},
  \qquad
  u_\pm(K)=2\log x_\pm(K).
\]
Then the feasible set is
\[
  \mathcal I_K(H)=[u_-(K),u_+(K)]
\]
and
\[
  D_{K,\diag}^{2D}(u_0;H)
  =
  \frac{1}{\sqrt2}\dist(u_0,\mathcal I_K(H)).
\]
\end{theorem}

The interval \(\mathcal I_K(H)\) records the exact expressive range of the
remaining log-scale ratio.  The off-diagonal entry of (H) enters through
\(\det H\), so correlation shrinks the feasible interval and raises the
smallest reachable condition number.

Let
\[
  q=P_{\mathcal I_K(H)}(u_0)
\]
be the unique RDGC endpoint and fix \(T>0\).  For mechanism intensities
\(a=(a_1,a_2)\in\R^2\), consider paths \(v\in H^1([0,T])\) satisfying
\(v(0)=u_0\), \(v(T)=q\), with action
\begin{equation}
  \mathfrak A(v;a)
  =\frac14\int_0^T\dot v(t)^2\,\dd t
  -a_1\int_0^T v(t)\,\dd t
  -a_2\int_0^T\frac{t}{T}v(t)\,\dd t.
  \tag{D1}
\end{equation}

\begin{theorem}[Exact forced diagonal RDGC response]
\label{thm:diag-forced-response}
Assume \(K\ge K_{\diag}^\ast(H)\).  The action in \textup{(D1)} has the
unique minimizer
\[
  v_a(t)
  =u_0+\frac{q-u_0}{T}t+a_1w_1(t)+a_2w_2(t),
\]
where
\[
  w_1(t)=t(T-t),
  \qquad
  w_2(t)=\frac{Tt}{3}-\frac{t^3}{3T}.
\]
At \(a=0\), this is the constant-speed RDGC path and
\[
  \mathfrak A(v_0;0)
  =\frac{D_{K,\diag}^{2D}(u_0;H)^2}{2T}.
\]
For \(\Val(a)=\min_v\mathfrak A(v;a)\), the path response is
\(\partial_{a_i}v_a=w_i\), and
\begin{equation}
  \boxed{
  D_a^2\Val
  =-T^3
  \begin{pmatrix}
    1/6 & 1/12\\
    1/12 & 2/45
  \end{pmatrix}\prec0.}
  \tag{D2}
\end{equation}
The positive matrix inside the minus sign has determinant \(1/2160\).
\end{theorem}

The preceding examples use the unrestricted path class inside a geodesically
convex diagonal family.  A coordinate-update protocol gives a simple case in
which the admissible path geometry cannot be replaced by the ambient AIRM
projection in \cref{prop:rdgc-target-projection}.  Define
\[
  D_3(x)=\diag(e^{x_1},e^{x_2},e^{-x_1-x_2}),
  \qquad x\in\R^2,
\]
and call a path sequential if at most one of \(\dot x_1,\dot x_2\) is nonzero
almost everywhere.  This models a diagonal preconditioner that rescales one
independent log-coordinate at a time.

\begin{proposition}[Exact coordinate-sequential RDGC]
\label{prop:sequential-diagonal-rdgc}
For \(H=I\), put \(k=\log K\) and
\[
  \mathcal H_k
  =\left\{y\in\R^2:
  \max(y_1,y_2,-y_1-y_2)-
  \min(y_1,y_2,-y_1-y_2)\le k\right\}.
\]
The sequential path distance and RDGC are
\begin{equation}
  d_{\rm seq}(x,y)=\sqrt2\norm{x-y}_1,
  \qquad
  D_{K,{\rm seq}}^{3D}(x_0;I)
  =\sqrt2\min_{y\in\mathcal H_k}\norm{x_0-y}_1.
  \tag{D7}
\end{equation}
In particular, if \(x_0=(a,a)\) with \(a>k/3\), the unique symmetric target
point is \(q=(k/3,k/3)\), and
\begin{equation}
\begin{aligned}
  D_{K,{\rm seq}}^{3D}(x_0;I)
  &=2\sqrt2\left(a-\frac{k}{3}\right),\\
  d_{\AI}(D_3(x_0),D_3(q))
  &=\sqrt6\left(a-\frac{k}{3}\right).
\end{aligned}
\tag{D8}
\end{equation}
Thus the sequential update protocol has a strictly larger exact RDGC, by the
factor \(2/\sqrt3\), than the ambient AIRM projection benchmark.
\end{proposition}

The fixed-endpoint model verifies the unconstrained Green response directly
and shows that two time profiles acquire a nonzero interaction solely through
relaxation of a structured metric path.  The next model keeps the terminal
condition hard and moves its active spectral boundary.

Let \(b\in\R^m\), put \(\rho(u)=\rho_0+b^\top u\), and define the
determinant-one Hessian family
\[
  H_u=\diag(e^{\rho(u)/2},e^{-\rho(u)/2}).
\]

\begin{theorem}[Exact moving hard-target diagonal response]
\label{thm:diag-moving-hard-target}
Let \(K>1\), and suppose an open set \(\mathcal U_0\supset[0,1]^m\)
satisfies
\begin{equation}
  v_0>\rho(u)+\log K,
  \qquad u\in\mathcal U_0.
  \tag{D3}
\end{equation}
For the metric path \(D(v(t))\), impose the hard terminal condition
\[
  \kappa_{\rm gen}(H_u,D(v(T)))\le K
\]
and minimize the kinetic action \(\frac14\int_0^T\dot v^2\dd t\).  Then the
active target endpoint, path, multiplier, and value are
\begin{equation}
\begin{aligned}
  q_u&=\rho(u)+\log K,\\
  v_u(t)&=v_0+\frac{q_u-v_0}{T}t,\\
  \lambda_u&=\frac{v_0-q_u}{2T}>0,\\
  \mathcal V_{\rm hard}(u)&=\frac{(v_0-q_u)^2}{4T}.
\end{aligned}
\tag{D4}
\end{equation}
Consequently,
\begin{equation}
  \partial_{u_i}v_u(t)=b_i\frac{t}{T},
  \qquad
  \partial_{u_i}\lambda_u=-\frac{b_i}{2T},
  \qquad
  \boxed{D_{uu}^2\mathcal V_{\rm hard}=\frac{bb^\top}{2T}.}
  \tag{D5}
\end{equation}
Every Boolean conditional pair effect on the cube is therefore
\begin{equation}
  \boxed{\zeta_{\{i,j\}\mid C}^{\rm hard}=\frac{b_i b_j}{2T}.}
  \tag{D6}
\end{equation}
Thus equally oriented target shifts have positive hard interaction and
oppositely oriented shifts have negative hard interaction.  The whole cube
remains on one simple active spectral stratum by \textup{(D3)}, so
\textup{(D4)}--\textup{(D6)} instantiate \textup{(HT4)}--\textup{(HT6)}.
\end{theorem}

\FloatBarrier

%% file: sections/06_verification.tex
\section{Deterministic Verification}
\label{sec:verification}

The source tree contains a finite-dimensional verification of identities that
are algebraically independent of the proofs.  A nodal discretization of a
strictly convex quartic path action is solved under two affine bulk/terminal
interventions.  The exact tridiagonal Hessian supplies every Jacobi solve, and
tensor Gauss--Legendre quadrature integrates the curvature over the
intervention square.

At \(N=40\) nodes and quadrature order \(q=10\), the four-vertex effect is
\[
  \zeta_{\{1,2\}}=-0.15744124517954516,
  \qquad
  I_{\rm Jac}=-0.15744124517876962.
\]
Their absolute discrepancy is
\(7.755\times10^{-13}\), and the maximum Green-reciprocity error is
\(8.327\times10^{-17}\).  The minimum sampled coordinate-Hessian eigenvalue
is \(0.1041677\), which certifies positive definiteness of this discrete
problem only; its value is basis- and mesh-scaled.

\begin{table}[t]
\caption{Deterministic mesh and quadrature sensitivity.  The effect changes
with the mesh, while the two discrete evaluations agree to approximately
machine precision.  The coordinate eigenvalue is not a discretization of a
mesh-independent coercivity constant.}
\label{tab:verification-sensitivity}
\centering
\begin{tabular}{rrrr}
\toprule
\(N\) & four-vertex effect & identity error & coordinate \(\lambda_{\min}\)\\
\midrule
10 & \(-1.7647773\times10^{-1}\) & \(9.398\times10^{-13}\) & \(3.963001\times10^{-1}\)\\
20 & \(-1.6369020\times10^{-1}\) & \(1.123\times10^{-12}\) & \(2.049008\times10^{-1}\)\\
40 & \(-1.5744125\times10^{-1}\) & \(7.755\times10^{-13}\) & \(1.041677\times10^{-1}\)\\
80 & \(-1.5435403\times10^{-1}\) & \(4.931\times10^{-13}\) & \(5.251650\times10^{-2}\)\\
\midrule
\(q\) at \(N=40\) & 4 & 6 & 10 / 16\\
identity error & \(1.008\times10^{-10}\) & \(1.267\times10^{-12}\) &
\(7.755/7.507\times10^{-13}\)\\
\bottomrule
\end{tabular}
\end{table}

The same program verifies the diagonal interaction matrix and determinant in
\textup{(D2)}, both signs of the moving hard-target interaction in
\textup{(D6)}, and the sequential-to-ambient ratio in \textup{(D8)}.  A
separate symbolic check verifies the scalar
implicit-differentiation specialization of the third-order formula; it is not
a numerical validation of the full path-space recursion in
\cref{cor:all-order-mobius-jacobi}.

The deterministic entry point is
\nolinkurl{experiments/toy/verify_mobius_jacobi.py}; passing
\texttt{--sensitivity} produces \cref{tab:verification-sensitivity}.  No
random seed is used.  The reported run used Python 3.11.4, NumPy 1.26.4,
SciPy 1.10.1, and SymPy 1.11.1.  These checks compare two evaluations of the
same discretized identities.  They establish discrete algebraic consistency,
with no continuous-discretization error or optimizer-performance claim.

%% file: sections/07_discussion.tex
\section{Discussion}
\label{sec:discussion}

The results separate four logical layers of restricted metric optimization.
Path concatenation gives the exact min-plus semigroup.  Tonelli assumptions
give minimizers.  A compact-control representation supplies the conditional
Hamilton--Jacobi layer.  Smooth nondegeneracy gives the Jacobi inverse and its
Schur-complement response.  In the Hadamard mechanical model, nonpositive
curvature and convex potentials derive that nondegeneracy globally on an
entire neighborhood of the intervention cube.

The fixed-horizon energy gauge is essential.  It removes the
reparameterization zero mode of length while preserving the RDGC minimizer and
fixes the action value at \(D_{K,\F}^2/(2T)\).  The resulting Green operator
handles bulk and terminal Robin forcing in one system, generates the value
Hessian with exact spectral and rank information, and recursively determines
every prescribed finite interaction order under corresponding regularity.
When the condition target itself moves, the
bordered Jacobi--KKT operator adds the active boundary and its multiplier to
the same response calculation.  This supplies a classical hard-RDGC response
on each regular active spectral stratum.

The strongest statements are finite-dimensional and deterministic.  A
nonlinear training problem requires a separate realization theorem connecting
optimizer state, damping, bias correction, and interpolation to an admissible
metric path.  Eigenvalue collisions, nonsmooth active-set changes, rank
changes, multiple minimizers, and loss of Jacobi coercivity require stratified
or generalized response theory.  These boundaries delimit the classical
derivatives proved here.

In summary, RDGC is a genuine path-space value: its global elimination law is
Bellman composition, its smooth local response is a Jacobi Schur complement,
and its finite intervention effects are exact integrals of a Green or bordered
Green kernel.  The closed-form diagonal model demonstrates the unconstrained
layers within a restricted preconditioning geometry.  Its moving-Hessian
specialization additionally gives the hard spectral endpoint, multiplier, and
both possible pair-interaction signs in closed form.  The coordinate-sequential
three-dimensional example shows separately that an update protocol can make
RDGC strictly exceed the ambient static projection benchmark.

%% file: appendix/a_path_proofs.tex
\section{Proofs for Path-Space Reduction}
\label{app:path-proofs}

\begin{proof}[Proof of \cref{thm:path-dpp}]
Fix an admissible path \(\gamma\in\Path_{s,t}(x,z)\) and put
\(y=\gamma(r)\).  Restriction and action additivity give
\[
  \mathcal A_{s,t}(\gamma)
  \ge
  c_{s,r}(x,y)+c_{r,t}(y,z).
\]
Taking the infimum over \(\gamma\) proves that the left side of (P1) is at
least the right side.

If the right side is \(+\infty\), the preceding inequality already forces the
left side to be \(+\infty\).  Otherwise, for every \(\varepsilon>0\), choose
\(y\) whose two finite transition costs have sum within \(\varepsilon\) of
the right side, and choose segments \(\gamma_1,\gamma_2\) whose actions are
within \(\varepsilon\) of those costs.  Their concatenation is admissible and
has action at most
\[
  \inf_{w}\{c_{s,r}(x,w)+c_{r,t}(w,z)\}+3\varepsilon.
\]
Let \(\varepsilon\downarrow0\).  The interval lower bounds in
\cref{ass:path-composition} exclude \(-\infty\) transition costs, so every
sum used above is defined.
\end{proof}

\begin{proof}[Proof of \cref{thm:global-mobius-jacobi}]
We divide the argument into the global variational, differential, and Boolean
layers.

\paragraph{Path norm and action regularity.}
For \(\xi\in\mathcal V_u\), parallel transport to one tangent space and the
fundamental theorem of calculus give
\begin{equation}
  \norm{\xi(T)}^2
  \le T\int_0^T\norm{D_t\xi}^2\dd t,
  \qquad
  \int_0^T\norm{\xi(t)}^2\dd t
  \le\frac{T^2}{2}\int_0^T\norm{D_t\xi}^2\dd t.
  \tag{A.P1}
\end{equation}
Hence the derivative norm is equivalent to the usual \(H^1\) norm on the
fixed-initial fluctuation space.  The hypotheses of
\cref{lem:path-action-smoothness} hold with \(k=r+2\); consequently
\((u,\gamma)\mapsto\mathcal A_u(\gamma)\) is a \(C^{r+1}\) functional in
every exponential path chart used below.

\paragraph{Coercivity and existence without a potential lower bound.}
Geodesic convexity gives the global support inequalities
\[
  W_u(t,x)
  \ge W_u(t,x_0)
  +\ip{\operatorname{grad}W_u(t,x_0)}{\log_{x_0}x},
\]
\[
  \Phi_u(x)
  \ge \Phi_u(x_0)
  +\ip{\operatorname{grad}\Phi_u(x_0)}{\log_{x_0}x}.
\]
The uniform base-point bounds imply that both right sides are bounded below
by \(-C(1+d(x_0,x))\), uniformly in \((t,u)\).  Put
\(E(\gamma)=\int_0^T\norm{\dot\gamma}^2\dd t\).  Since
\[
  d(x_0,\gamma(t))\le\sqrt t\,E(\gamma)^{1/2},
\]
there are constants \(C_1,C_2\), independent of \(u\in\mathcal U_0\), such
that
\begin{equation}
  \mathcal A_u(\gamma)
  \ge\frac12E(\gamma)-C_1E(\gamma)^{1/2}-C_2.
  \tag{A.P2}
\end{equation}
Thus minimizing sequences have uniformly bounded \(H^1\) energy even when a
potential tends linearly to \(-\infty\).

The common initial point and the bound
\[
  d(\gamma(s),\gamma(t))
  \le |t-s|^{1/2}E(\gamma)^{1/2}
\]
confine a minimizing sequence to a common bounded ball and make it
equicontinuous.  Finite-dimensional Hadamard manifolds are proper, so
Arzela--Ascoli gives a uniformly convergent subsequence.  In a finite family
of path charts its velocities converge weakly in \(L^2\).  Kinetic energy is
weakly lower semicontinuous, while uniform path convergence and local uniform
continuity of the potentials give convergence of the bulk and terminal
potential terms.  The limit therefore attains the infimum.

\paragraph{Strict path convexity and uniqueness.}
For two admissible paths \(\gamma^0,\gamma^1\), let
\(\Gamma(s,t)\) be their pointwise geodesic homotopy and write
\(H=\partial_s\Gamma\), \(V=\partial_t\Gamma\).  The common initial point
gives \(H(s,0)=0\).  The second variation along the homotopy is
\[
\begin{aligned}
  \frac{\dd^2}{\dd s^2}\mathcal A_u(\Gamma_s)
  ={}&\int_0^T
  \left[\norm{D_tH}^2
  -\ip{R(H,V)V}{H}
  +\operatorname{Hess}W_u[H,H]\right]\dd t\\
  &+\operatorname{Hess}\Phi_u[H(T),H(T)].
\end{aligned}
\]
Every term after \(\norm{D_tH}^2\) is nonnegative by nonpositive curvature
and (MJ2).  If the lower bound vanishes, then \(D_tH=0\); together with
\(H(s,0)=0\), this implies \(H\equiv0\).  Hence the action is strictly convex
along every nontrivial path homotopy.  The minimizer is unique.

\paragraph{Euler--Lagrange equation and index form.}
For \(\xi(0)=0\), the first variation at the minimizer is
\[
\begin{aligned}
  D\mathcal A_u(\gamma_u)[\xi]
  ={}&\int_0^T
  \ip{-D_t\dot\gamma_u+\operatorname{grad}W_u}{\xi}\dd t\\
  &+\ip{\dot\gamma_u(T)+\operatorname{grad}\Phi_u}{\xi(T)}.
\end{aligned}
\]
Compactly supported tests and then arbitrary terminal values give (MJ3) in
weak form.  In a coordinate chart the weak equation first gives
\(\dot\gamma_u\in W^{1,1}\), hence \(\gamma_u\in C^1\); continuity in
\(t\) of the force and smoothness of the connection then give
\(\gamma_u\in C^2\) and (MJ3) pointwise.  Differentiating once more gives
(MJ4).  Nonpositive
curvature and (MJ2) give (MJ5) term by term.  The inequalities in (A.P1) show
that \(\mathfrak B_u\) is continuous and coercive on \(\mathcal V_u\).
Lax--Milgram therefore makes
\(\mathfrak J_u:\mathcal V_u\to\mathcal V_u^*\) an isomorphism, with inverse
norm at most one in the derivative norm.

\paragraph{Smooth global branch.}
In a path chart, write the first-order condition as
\[
  \mathscr F(u,\gamma)=D_\gamma\mathcal A_u(\gamma)=0.
\]
Its path derivative at \(\gamma_u\) is \(\mathfrak J_u\), hence invertible.
The Hilbert implicit-function theorem produces a local \(C^r\) critical
branch around every \(u\in\mathcal U_0\).  Strict convexity makes every
critical path the unique global minimizer, so local branches agree on overlaps
and form one \(C^r\) map on the whole cube neighborhood.

\paragraph{Weak and strong Green--Jacobi response.}
Differentiating \(\mathscr F(u,\gamma_u)=0\) in coordinate \(u_i\) gives
\[
  \mathfrak J_u(\partial_{u_i}\gamma_u)+\mathfrak f_i^u=0,
\]
which proves (MJ6).  For smooth \(\xi,\eta\), covariant integration by parts
in (MJ4) gives
\begin{equation}
\begin{aligned}
  \mathfrak B_u(\xi,\eta)
  ={}&\int_0^T\ip{J_u\xi}{\eta}\dd t\\
  &+\ip{D_t\xi(T)+\operatorname{Hess}\Phi_u\,\xi(T)}{\eta(T)}.
\end{aligned}
\tag{A.P3}
\end{equation}
The closed coercive form corresponds to the Friedrichs self-adjoint operator
with homogeneous conditions
\(\xi(0)=0\) and
\(D_t\xi(T)+\operatorname{Hess}\Phi_u\,\xi(T)=0\).  Comparing
\(\mathfrak B_u(\chi_i^u,\eta)=\mathfrak f_i^u[\eta]\) with (A.P3), first for
compactly supported tests and then for arbitrary endpoint traces, yields the
bulk equation and nonhomogeneous Robin condition in (MJ7).  One-dimensional
elliptic regularity gives the unique classical response.  This also identifies
the weak inverse, Friedrichs resolvent, and Green response.

\paragraph{Value curvature and reciprocity.}
The envelope identity gives
\[
  \partial_{u_i}\mathcal V=\Delta_i(\gamma_u).
\]
Differentiating and using (MJ6) yields
\[
  \partial_{u_i u_j}^2\mathcal V
  =-\mathfrak f_i^u[\chi_j^u]
  =-\mathfrak B_u(\chi_i^u,\chi_j^u).
\]
Symmetry of \(\mathfrak B_u\) proves reciprocity and (MJ8).  Applying the
identity to a linear combination of interventions gives the negative
semidefinite Gram representation.  Under the Riesz identification, the
spectral bounds
\[
  \norm{\mathfrak J_u}^{-1}I
  \preceq\mathfrak J_u^{-1}
  \preceq\norm{\mathfrak J_u^{-1}}I
\]
give (MJ8a).  Since \(\mathfrak J_u^{-1}\) is strictly positive,
\(\langle f,\mathfrak J_u^{-1}f\rangle=0\) holds exactly when \(f=0\).
This proves the kernel and rank identities in (MJ8b).

\paragraph{Boolean integration.}
For a smooth scalar function, the coordinate difference satisfies
\[
  f(u+e_i)-f(u)=\int_0^1\partial_{u_i}f(u+se_i)\dd s
\]
when \(u_i=0\).  Iterating this identity over \(i\in S\) proves (MJ9).
Substitution of (MJ8) proves (MJ10).  Starting the same integration from
\(\mathbf1_C\) proves the conditional-face formula.
\end{proof}

\begin{lemma}[Smoothness of the mechanical action in a path chart]
\label{lem:path-action-smoothness}
Let \(\bar\gamma\in H^1([0,T],\mathcal N)\), and use a smooth exponential
path chart \(\gamma=\exp_{\bar\gamma}\xi\) on an \(H^1\)-neighborhood of
\(\bar\gamma\).  Suppose that the metric coefficients are smooth and that a
bulk potential \(W_u(t,x)\) is continuous in \(t\), \(C^k\) in \((u,x)\),
and has state and parameter derivatives through order \(k\) locally uniform
in \(t\).  Then
\[
  (u,\xi)\longmapsto
  \int_0^T\left(\frac12\norm{\dot\gamma}^2
  +W_u(t,\gamma(t))\right)\dd t
\]
is \(C^{k-1}\) on a sufficiently small path-chart neighborhood.  A
\(C^k\) terminal potential composed with the endpoint trace has the same
regularity.  Its derivatives are the variational derivatives obtained by
differentiating under the integral and at the endpoint.
\end{lemma}

\begin{proof}
The embedding \(H^1([0,T])\hookrightarrow C^0([0,T])\) is continuous, and
one-dimensional \(H^1\) is a Banach algebra.  A small chart ball therefore
maps into one compact tube around \(\bar\gamma([0,T])\).  On that tube, all
coefficient derivatives used below are uniformly bounded.  In local
coordinates the kinetic integrand is
\(g_{ab}(\gamma)\dot\gamma^a\dot\gamma^b/2\).  Repeated differentiation
produces finite sums of bounded Nemytskii coefficients multiplied by at most
two \(L^2\) velocity factors and by \(H^1\) variations.  H\"older's inequality
and the algebra estimate make every derivative through order \(k-1\)
continuous in the \(H^1\) norm.  The potential terms follow from the same
Nemytskii estimate with no velocity factors.  Finally, endpoint evaluation is
a continuous linear map from \(H^1\) to a finite-dimensional tangent space,
so ordinary finite-dimensional composition treats the terminal term.
\end{proof}

\begin{proof}[Proof of \cref{thm:hard-target-jacobi}]
Work in a smooth path chart around \(\gamma_0\) and identify nearby
fixed-initial fluctuation spaces with one Hilbert space \(\mathcal V\).  The
KKT map is
\[
  \mathscr F(u,\gamma,\lambda)
  =\bigl(D_\gamma\mathscr L(u,\gamma,\lambda),
  c_u(\gamma(T))\bigr).
\]
Its derivative in \((\gamma,\lambda)\) at the reference point is precisely
\(\mathfrak K_0\) in \textup{(HT3)}.

We first prove that this derivative is invertible.  Coercivity makes
\(\mathfrak J_0:\mathcal V\to\mathcal V^*\) a positive self-adjoint
isomorphism.  Since \(\mathfrak B_0\ne0\),
\[
  \sigma_0
  =\mathfrak B_0\mathfrak J_0^{-1}\mathfrak B_0^*>0.
\]
For prescribed \((f,a)\in\mathcal V^*\times\R\), the system
\[
  \mathfrak J_0\xi+\mathfrak B_0^*\beta=f,
  \qquad \mathfrak B_0\xi=a
\]
has the unique solution
\[
  \beta=\sigma_0^{-1}
  \bigl(\mathfrak B_0\mathfrak J_0^{-1}f-a\bigr),
  \qquad
  \xi=\mathfrak J_0^{-1}(f-\mathfrak B_0^*\beta).
\]
Hence \(\mathfrak K_0\) is an isomorphism.  The Hilbert implicit-function
theorem gives the unique \(C^r\) KKT branch.  Strict positivity of the
multiplier and coercivity persist after shrinking the parameter
neighborhood.

For a feasible path \(\gamma\) near \(\gamma_u\), Taylor expansion of the
Lagrangian and stationarity give
\[
  \mathcal A_u(\gamma)-\mathcal A_u(\gamma_u)
  =\mathscr L(u,\gamma,\lambda_u)
   -\mathscr L(u,\gamma_u,\lambda_u)
   -\lambda_u c_u(\gamma(T)).
\]
The first difference is positive quadratic to leading order by coercivity,
and the last term is nonnegative because \(\lambda_u>0\) and
\(c_u(\gamma(T))\le0\).  After one uniform neighborhood reduction,
\(\gamma_u\) is therefore the unique constrained local minimizer there.

Differentiating \(\mathscr F(u,\gamma_u,\lambda_u)=0\) gives
\[
  \mathfrak K_u\partial_{u_i}(\gamma_u,\lambda_u)
  +\mathfrak b_i^u=0,
\]
which proves \textup{(HT4)}.  Because the active constraint vanishes,
\[
  \mathcal V_{\rm hard}(u)
  =\mathscr L(u,\gamma_u,\lambda_u).
\]
The envelope identity and a second differentiation yield
\[
  \partial_{u_i u_j}^2\mathcal V_{\rm hard}
  =\partial_{u_i u_j}^2\mathscr L
  +\left\langle
  \mathfrak b_i^u,\partial_{u_j}(\gamma_u,\lambda_u)
  \right\rangle.
\]
Substitution of \textup{(HT4)} proves \textup{(HT5)}.  Self-adjointness of
the bordered map gives reciprocity.  If both data are affine in \(u\), the
direct second derivative is zero; applying the coordinate fundamental theorem
of calculus twice on the indicated face proves \textup{(HT6)}.
\end{proof}

\begin{proof}[Proof of \cref{cor:all-order-mobius-jacobi}]
Apply the set-indexed multivariate Faa di Bruno formula to
\(\mathscr F(u,\gamma_u)=0\).  Since \(\mathcal A_u\), and hence
\(\mathscr F\), is affine in \(u\), at most one index can act on the explicit
\(u\)-argument.  The resulting expansion is
\[
  0=\sum_{\substack{S\subseteq I,\ |S|\le1\\
  \pi\in\Pi(I\setminus S)}}
  D_u^S D_\gamma^{|\pi|}\mathscr F
  \bigl[\eta_B^u\bigr]_{B\in\pi}.
\]
The unique term containing the highest derivative \(\eta_I^u\) corresponds to
\(S=\varnothing\), \(\pi=\{I\}\), and equals
\(\mathfrak J_u\eta_I^u\).  Moving all other terms to the right and applying
\(\mathfrak J_u^{-1}\) proves (MJ11)--(MJ12).  Thus every derivative is
generated recursively from lower-order derivatives by the same Green inverse.

For completeness, applying the same partition formula to the scalar composite
\(\mathcal A_u(\gamma_u)\) gives
\begin{equation}
  \partial_I^{|I|}\mathcal V
  =\sum_{\substack{S\subseteq I,\ |S|\le1\\
  \pi\in\Pi(I\setminus S)}}
  D_u^S D_\gamma^{|\pi|}\mathcal A_u
  \bigl[\eta_B^u\bigr]_{B\in\pi},
  \tag{A.P4}
\end{equation}
where the term \(D_\gamma\mathcal A_u[\eta_I^u]\) vanishes by stationarity.
Equations (A.P4) and (MJ9) give every higher-order Boolean effect.

For three distinct indices, differentiate
\(\mathfrak J_u\partial_{u_j}\gamma_u+D\Delta_j=0\) in \(u_k\), use the
result to eliminate \(\partial_{u_j u_k}^2\gamma_u\) from the second derivative
of \(\partial_{u_i}\mathcal V=\Delta_i(\gamma_u)\), and substitute
\(\partial_{u_l}\gamma_u=-\chi_l\).  Symmetry of \(\mathfrak J_u\) cancels
the second path-response terms and leaves exactly (MJ13).
\end{proof}

\begin{proof}[Proof of \cref{cor:spd-mobius-jacobi}]
The AIRM cone is Hadamard, and a structured realization satisfying the stated
hypotheses has the same intrinsic nonpositive-curvature property.  The
information potentials supply (MJ2), so
\cref{thm:global-mobius-jacobi,cor:all-order-mobius-jacobi} apply directly.
The displayed operator is (MJ7) with the AIRM connection and curvature.  Its
curvature term is nonnegative in the index form, and (MJ10) gives the pair
formula.  The final scope statement records exactly where the classical
Hessian used by the theorem ceases to exist.
\end{proof}

\begin{proof}[Proof of \cref{thm:length-energy-gauge}]
For every finite-energy path, Cauchy--Schwarz gives
\[
  \frac12\int_0^T\norm{\dot\gamma_t}^2\dd t
  \ge
  \frac1{2T}\left(\int_0^T\norm{\dot\gamma_t}\dd t\right)^2
  =\frac{\Len(\gamma)^2}{2T}
  \ge\frac{D(x,\mathcal C)^2}{2T}.
\]
Thus \(\mathcal E_T\ge D^2/(2T)\).  If \(D<+\infty\), choose paths with
length converging to \(D\) and reparameterize each at constant speed.  The
reparameterized paths remain admissible and have energy
\(\Len(\gamma)^2/(2T)\), proving the reverse inequality after taking the
limit.  If \(D=+\infty\), a finite-energy path would have finite length by
Cauchy--Schwarz, a contradiction, so both sides are \(+\infty\).

Equality in Cauchy--Schwarz holds exactly when the metric speed is constant
almost everywhere.  Equality in the second inequality holds exactly for a
length minimizer.  These two equality conditions characterize the energy
minimizers and prove the RDGC specialization.
\end{proof}

\begin{proof}[Proof of \cref{cor:path-bellman}]
Apply \cref{thm:path-dpp} to \(c_{s,T}(x,z)\), add \(\Gamma(z)\), and regroup
the two successive infima:
\[
\begin{aligned}
V(s,x)
&=\inf_z\inf_y
\{c_{s,r}(x,y)+c_{r,T}(y,z)+\Gamma(z)\}\\
&=\inf_y\{c_{s,r}(x,y)+V(r,y)\}.
\end{aligned}
\]
\end{proof}

\begin{proof}[Proof of \cref{thm:path-tonelli}]
Take a minimizing sequence with uniformly bounded objective.  The lower bound
on \(\overline L\) and the lower bound on \(\Gamma\) give a
uniform \(L^p\) bound on path speeds.  Since \(p>1\), H\"older's inequality
gives uniform equicontinuity.  The common initial point and the speed bound
confine every path to a common closed bounded ball.  Finite-dimensional
Hadamard manifolds are proper, so Arzel\`a--Ascoli gives a uniformly convergent
subsequence.  The corresponding velocities converge weakly after passage to
  local charts and a further subsequence.  Sequential viability keeps the limit
  admissible.  More explicitly, subdivide the compact limiting path into
  finitely many intervals whose images lie in normal coordinate charts.  For
  all sufficiently large indices the approximating segments lie in the same
  charts; their coordinate paths converge strongly and uniformly, while their
  coordinate velocities converge weakly in \(L^p\).  On each subinterval the
  normal-integrand hypothesis and convexity in velocity give weak lower
  semicontinuity of the integral functional
  \cite[Chapter~3]{dacorogna2008direct}.  Chart-boundary times form a finite
  measure-zero set, so summing the chartwise inequalities gives the global
  lower bound.  Therefore
\[
  \mathcal A(\gamma)
  \le
  \liminf_n\mathcal A(\gamma_n),
\]
and lower semicontinuity of \(\Gamma\) treats the endpoint.  Hence the limit
attains the value.

If two different minimizing paths existed, their pointwise geodesic
interpolation would be admissible and strict geodesic convexity would give a
strictly smaller objective, a contradiction.
\end{proof}

\begin{proof}[Proof of \cref{prop:path-hjb}]
The Bellman identity on \([t,t+h]\) gives
\[
  V(t,G)=
  \inf_{a_{[t,t+h]}}
  \left\{
  \int_t^{t+h}\ell(s,\gamma_s,a_s)\,\dd s
  +V(t+h,\gamma(t+h))
  \right\}.
\]
  Put
  \[
    q_\varphi(s,X,a)
    =\ell(s,X,a)+\dd_X\varphi(s,X)[b(s,X,a)].
  \]
  Let a smooth test function \(\varphi\) touch \(V\) from above at \((t,G)\).
  For every constant control \(a\in\mathcal U\), the Bellman inequality and the
  local maximum of \(V-\varphi\) give
  \[
    0\le
    \int_t^{t+h}\ell(s,\gamma_s,a)\,\dd s
    +\varphi(t+h,\gamma(t+h))-\varphi(t,G).
  \]
  Uniform local Lipschitz continuity of \(b\) gives
  \(\gamma(t+h)=\exp_G(hb(t,G,a)+o(h))\), uniformly in \(a\).  Divide by
  \(h\), let \(h\downarrow0\), and use compactness of \(\mathcal U\).  The
  inequality holds for every \(a\), hence
  \[
    -\partial_t\varphi(t,G)
    +\mathcal H(t,G,-\dd_G\varphi(t,G))\le0,
  \]
  which is the viscosity subsolution inequality for the sign convention in
  (P4).

  If \(\varphi\) touches \(V\) from below, choose a measurable control whose
  short-horizon cost is within \(o(h)\) of the Bellman infimum.  The local
  minimum of \(V-\varphi\) then yields
  \[
    o(h)\ge
    \int_t^{t+h}\ell(s,\gamma_s,a_s)\,\dd s
    +\varphi(t+h,\gamma(t+h))-\varphi(t,G).
  \]
  Along this trajectory, the chain rule rewrites the right side as
  \[
    \int_t^{t+h}
    \bigl[\partial_s\varphi(s,\gamma_s)
    +q_\varphi(s,\gamma_s,a_s)\bigr]\,\dd s.
  \]
  The trajectory remains uniformly close to \(G\) on the shrinking interval.
  Uniform continuity on the resulting compact neighborhood gives, pointwise in
  the measurable control,
  \[
    q_\varphi(s,\gamma_s,a_s)
    \ge \inf_{a\in\mathcal U}q_\varphi(t,G,a)-o(1).
  \]
  After division by \(h\), this proves
  \[
    -\partial_t\varphi(t,G)
    +\mathcal H(t,G,-\dd_G\varphi(t,G))\ge0.
  \]
  No closure or convexification of averaged velocity--cost pairs is used.
  Continuity of \(V\) and \(\Gamma\) supplies the terminal condition.  At a
  differentiability point, the two inequalities coincide.  The feedback
  formula uses an attaining maximizer, which exists by compactness.  If
  comparison holds in the declared class, standard comparison gives
  uniqueness.
\end{proof}

\begin{proof}[Proof of \cref{thm:path-schur}]
The first-order condition is
\[
  D_\eta\mathfrak A(z,\eta_\star(z))=0.
\]
Coercivity and self-adjointness imply that \(\Jac:\Hilb_0\to\Hilb_0\) is
invertible; equivalently one may apply Lax--Milgram to the second-variation
form.  The Hilbert-space implicit-function theorem therefore produces a unique
\(C^2\) critical section near \((z_0,\eta_0)\).  Coercivity persists after
shrinking the neighborhood, so this section remains the unique local
minimizer.  Differentiating the first-order condition gives
\[
  \mathfrak A_{\eta z}+\Jac D\eta_\star=0,
\]
which proves (P5).

The envelope identity is
\[
  D\Val
  =
  \mathfrak A_z+\mathfrak A_\eta D\eta_\star
  =
  \mathfrak A_z.
\]
Differentiating again and substituting (P5) gives
\[
  D^2\Val
  =
  \mathfrak A_{zz}
  -\mathfrak A_{z\eta}\Jac^{-1}\mathfrak A_{\eta z},
\]
which is (P6).
\end{proof}

\begin{proof}[Proof of \cref{cor:rdgc-jacobi-bridge}]
The classical second-variation formula for kinetic energy along a geodesic
with fixed endpoints is the displayed index form.  Since \(\F\) is totally
geodesic, its intrinsic connection and curvature are the restrictions used by
that formula.  Nonpositive sectional curvature gives
\[
  \ip{R(\dot\gamma_0,h)h}{\dot\gamma_0}\le0,
\]
and hence
\(I(h,h)\ge\int_0^T\norm{D_th}^2\dd t\).  The derivative seminorm is a norm
on \(H_0^1\), equivalent to its standard Sobolev norm by Poincare's inequality.
Thus the Riesz representative of \(I\) is bounded, self-adjoint, and coercive.
The length--energy identification and the final sensitivity claim now follow
from \cref{thm:length-energy-gauge,thm:path-schur}.
\end{proof}

\begin{proof}[Proof of \cref{cor:path-schur-associativity}]
Let \(q(\delta z,\delta\eta_1,\delta\eta_2)\) be the quadratic form defined by
the complete Hessian.  Coercivity gives unique quadratic minimizers.  Hence
\[
  \inf_{\delta\eta_1,\delta\eta_2}q
  =
  \inf_{\delta\eta_1}
  \left(\inf_{\delta\eta_2}q\right).
\]
Each quadratic elimination is represented by the corresponding Schur
complement.  Equality for every \(\delta z\) proves (P7).
\end{proof}

\begin{proof}[Proof of \cref{cor:path-interaction}]
Affineness in \(u\) gives
\[
  \mathfrak A_{uu}=0,
  \qquad
  \mathfrak A_{\eta u}=\mathcal G.
\]
The \(u\)-block of (P6) is therefore
\[
  D^2_{uu}\Val=-\mathcal G^*\Jac^{-1}\mathcal G.
\]
For any \(a\in\R^m\),
\[
  \ip{a}{D^2_{uu}\Val\,a}
  =
  -\ip{\mathcal Ga}{\Jac^{-1}\mathcal Ga}\le0,
\]
which proves (P8) and (P9).  Applying the one-dimensional fundamental theorem
of calculus successively in coordinates \(u_i\) and \(u_j\) gives
\[
  \Val(e_i+e_j)-\Val(e_i)-\Val(e_j)+\Val(0)
  =
  \int_0^1\!\int_0^1
  \partial^2_{u_i u_j}\Val(se_i+te_j)\,\dd s\,\dd t.
\]
Substitution of (P9) proves (P10).
\end{proof}

%% file: appendix/b_spd_diagonal_proofs.tex
\section{Proofs for the SPD Specialization and Diagonal Model}
\label{app:spd-diagonal-proofs}

\begin{proof}[Proof of \cref{thm:full-spd-benchmark}]
Let \(S_1\in\Ck\), and let \(z_1\le\cdots\le z_d\) be its ordered
log-eigenvalues.  The affine-invariant distance satisfies the spectral lower
bound
\[
  d_{\AI}(S_0,S_1)\ge\norm{y-z}_2
\]
\cite{bhatia2007positive}.  Since \(S_1\in\Ck\), the vector \(z\) has
width at most \(\log K\), so \(z_i\in[c,c+\log K]\) for some \(c\).
Therefore
\[
  d_{\AI}(S_0,S_1)^2
  \ge
  \sum_{i=1}^d\dist(y_i,[c,c+\log K])^2.
\]

Conversely, fix \(c\) and set
\[
  z_i(c)=\Pi_{[c,c+\log K]}(y_i).
\]
Choose \(S_1\) with the same eigenvectors as \(S_0\) and log-eigenvalues
\(z_i(c)\).  Then \(S_1\in\Ck\), the matrices commute, and
\(d_{\AI}(S_0,S_1)=\norm{y-z(c)}_2\).  Minimizing over \(c\) proves the
formula.
\end{proof}

\begin{proof}[Proof of \cref{cor:hard-condition-response}]
Simplicity of the extreme generalized eigenvalues makes \(c_u|_{\F}\)
smooth near \((u_0,q_0)\).  For each fixed \(u\), geodesic convexity of the
log-condition number makes \(\mathcal T_u\) geodesically convex.  The AIRM
projection is therefore unique, and the fixed-horizon kinetic minimizer is the
constant-speed geodesic to that projection by
\cref{prop:rdgc-target-projection,thm:length-energy-gauge}.

Since \(G_0\) lies outside the target, the endpoint constraint is active.
Regularity of its restricted differential gives LICQ.  Projection
transversality has the form
\[
  \dot\gamma_{u_0}(T)
  +\lambda_0\operatorname{grad}^{\AI}(c_{u_0}|_{\F})(q_0)=0.
\]
The multiplier is positive: if it vanished, the terminal velocity and hence
the constant speed would vanish, contradicting
\(G_0\notin\mathcal T_{u_0}\).

The second variation of the endpoint Lagrangian is
\[
  \int_0^T\left(
  \norm{D_t\xi}^2
  -\ip{R^{\AI}(\xi,\dot\gamma)\dot\gamma}{\xi}
  \right)\dd t
  +\lambda_0\operatorname{Hess}^{\AI}(c_{u_0}|_{\F})
  [\xi(T),\xi(T)].
\]
Nonpositive curvature and geodesic convexity make the last two contributions
nonnegative, so this form is coercive in the fixed-initial derivative norm.
All hypotheses of \cref{thm:hard-target-jacobi} hold.  Its implicit branch is
the unique global projection branch because the endpoint objective is
strictly convex and \(\mathcal T_u\) is convex.  The value identity follows
from the length--energy gauge, and the response and Hessian formulas follow
from \textup{(HT4)}--\textup{(HT6)}.
\end{proof}

\begin{proof}[Proof of \cref{prop:rdgc-target-projection}]
The function \(G\mapsto\log\kappa(G)\) is geodesically convex on the SPD
cone, so \(\Ck\) is closed and geodesically convex.  Hence
\(\mathcal T_{K,\F}(H)\) is closed and geodesically convex.  A nonempty
closed geodesically convex subset of a finite-dimensional Hadamard manifold
has a unique metric projection.  The geodesic from \(S_0\) to this projection
stays inside \(\mathcal S_\F(H)\) and has length equal to the metric
distance; every admissible curve has length at least the distance between its
endpoints.  This proves the RDGC identity and uniqueness.

For a remaining horizon \(h=T-t\), Cauchy--Schwarz gives
\[
  \frac12\int_t^T\norm{\dot\gamma_s}^2\,\dd s
  \ge
  \frac{\Len(\gamma)^2}{2h}
  \ge
  \frac{d_{\AI}(S,\mathcal T_{K,\F}(H))^2}{2h}.
\]
The constant-speed geodesic to the projection attains equality.
\end{proof}

\begin{proof}[Proof of \cref{prop:spd-spectral-convexity}]
For \(A,B\in\SPD^d\), \(t\in[0,1]\), and \(1\le k\le d\), exterior powers
commute with the weighted geometric mean.  Monotonicity of that mean yields
\[
  \prod_{i=1}^k\lambda_i^\downarrow(A\#_tB)
  \le
  \prod_{i=1}^k
  \lambda_i^\downarrow(A)^{1-t}
  \lambda_i^\downarrow(B)^t.
\]
Equality holds for \(k=d\).  After taking logarithms,
\begin{equation}
  \log\lambda(A\#_tB)
  \prec
  (1-t)\log\lambda(A)+t\log\lambda(B).
  \tag{B.1}
\end{equation}

Fix \(H\succ0\), let \(G_t=G_0\#_tG_1\), and set
\(A_i=H^{-1/2}G_iH^{-1/2}\).  Congruence is an affine-invariant isometry and
preserves geometric means.  The relative eigenvalues
\(G^{-1/2}HG^{-1/2}\) are the reciprocals of those of
\(H^{-1/2}GH^{-1/2}\).  Multiplication by \(-1\) preserves majorization, and
a convex permutation-invariant function is Schur convex.  Applying
\textup{(B.1)} and ordinary convexity of \(\phi\) gives
\[
  \Omega_{\phi,H}(G_t)
  \le
  (1-t)\Omega_{\phi,H}(G_0)
  +t\Omega_{\phi,H}(G_1).
\]
Coordinatewise monotonicity followed by convexity of \(\rho\) proves the
claim for \(\mathcal J\).  Closedness follows from continuity of the
finite-dimensional convex spectral functions.  Finally,
\[
  \log\lambda((cG)^{-1/2}H(cG)^{-1/2})
  =
  \log\lambda(G^{-1/2}HG^{-1/2})-(\log c)\mathbf1,
\]
which proves scale invariance.
\end{proof}

\begin{proof}[Proof of \cref{thm:2d-diag}]
For \(D(u)=\diag(e^{u/2},e^{-u/2})\),
\[
  \log D(u)=\diag(u/2,-u/2),
\]
so the induced line element is \(\dd s^2=\frac12\,\dd u^2\).
Let
\[
  A(u)=D(u)^{-1/2}HD(u)^{-1/2}.
\]
Because \(\det D(u)=1\), one has \(\det A(u)=\det H\).  If the eigenvalues
of \(A(u)\) are \(\mu\le\nu\), the condition \(\nu/\mu\le K\) with
fixed product \(\mu\nu=\det H\) is equivalent to
\[
  \mu+\nu
  \le
  \left(\sqrt K+\frac1{\sqrt K}\right)\sqrt{\det H}
  =R_K.
\]
Moreover,
\[
  \tr A(u)=H_{11}e^{-u/2}+H_{22}e^{u/2}.
\]
With \(x=e^{u/2}\), the target condition is
\[
  H_{22}x^2-R_Kx+H_{11}\le0.
\]
Below \(K_{\diag}^\ast(H)\) the feasible set is empty.  Otherwise its roots
are \(x_\pm(K)\), so the feasible set in \(u\)-coordinates is
\([2\log x_-(K),2\log x_+(K)]\).  The one-dimensional line element gives
the claimed distance.
\end{proof}

\begin{proof}[Proof of \cref{thm:diag-forced-response}]
Write every admissible path as
\[
  v(t)=v_0(t)+\eta(t),
  \qquad
  v_0(t)=u_0+\frac{q-u_0}{T}t,
  \qquad
  \eta\in H_0^1(0,T).
\]
The vertical second variation is
\[
  D_{vv}^2\mathfrak A[h,k]
  =\frac12\int_0^T\dot h(t)\dot k(t)\,\dd t,
\]
which is coercive on \(H_0^1(0,T)\).  The weak Euler--Lagrange equation is
\[
  \frac12\int_0^T\dot v_a\dot h\,\dd t
  =
  a_1\int_0^T h\,\dd t
  +a_2\int_0^T\frac{t}{T}h\,\dd t.
\]
The affine path \(v_0\) has zero weak second derivative, while
\[
  -\frac12w_1''=1,
  \qquad
  -\frac12w_2''=\frac{t}{T},
  \qquad
  w_i(0)=w_i(T)=0.
\]
Thus the displayed \(v_a\) is the unique global minimizer.

At \(a=0\),
\[
  \mathfrak A(v_0;0)
  =\frac{(q-u_0)^2}{4T}
  =\frac{D_{K,\diag}^{2D}(u_0;H)^2}{2T}.
\]
The envelope identity gives
\[
  \partial_{a_i a_j}^2\Val
  =-\int_0^T f_i(t)w_j(t)\,\dd t,
  \qquad
  f_1(t)=1,
  \quad f_2(t)=t/T.
\]
Direct integration yields
\[
  \int_0^T w_1\,\dd t=\frac{T^3}{6},
  \qquad
  \int_0^T w_2\,\dd t
  =\int_0^T\frac{t}{T}w_1\,\dd t=\frac{T^3}{12},
\]
and
\[
  \int_0^T\frac{t}{T}w_2\,\dd t=\frac{2T^3}{45}.
\]
The determinant of the positive matrix in \textup{(D2)} is \(1/2160\), so
it is positive definite and \(D_a^2\Val\prec0\).
\end{proof}

\begin{proof}[Proof of \cref{thm:diag-moving-hard-target}]
For the diagonal Hessian and metric in the theorem,
\[
  D(v)^{-1/2}H_uD(v)^{-1/2}
  =\diag\left(e^{(\rho(u)-v)/2},e^{(v-\rho(u))/2}\right).
\]
Hence
\[
  \log\kappa_{\rm gen}(H_u,D(v))=|v-\rho(u)|,
\]
and the hard target in the \(v\)-coordinate is the interval
\([\rho(u)-\log K,\rho(u)+\log K]\).  Assumption \textup{(D3)} keeps
\(v_0\) strictly above this interval throughout \(\mathcal U_0\), so its
projection is the upper endpoint \(q_u\).  The unique kinetic minimizer is the
straight path in \textup{(D4)}, and direct integration gives the displayed
value.

On this active stratum use
\[
  c_u(v)=v-\rho(u)-\log K\le0.
\]
The endpoint first variation is
\(\frac12\dot v_u(T)\xi(T)\).  KKT transversality therefore reads
\[
  \frac12\dot v_u(T)+\lambda_u=0,
\]
which gives the positive multiplier in \textup{(D4)}.  Differentiating the
closed forms gives the path and multiplier responses in \textup{(D5)}.  Since
\(q_u\) is affine,
\[
  \partial_{u_i u_j}^2\mathcal V_{\rm hard}
  =\frac{b_i b_j}{2T}.
\]
This mixed derivative is constant on every Boolean face, so two applications
of the coordinate fundamental theorem of calculus prove \textup{(D6)}.
Finally, \(c_u\) is affine in \(u\), its endpoint differential is one, the
multiplier stays positive by \textup{(D3)}, and the kinetic Jacobi form is
coercive.  These are exactly the full-face hypotheses of
\cref{thm:hard-target-jacobi,cor:hard-condition-response}.
\end{proof}

\begin{proof}[Proof of \cref{prop:sequential-diagonal-rdgc}]
The AIRM line element in the log-coordinate chart is
\[
  \dd s^2
  =\dd x_1^2+\dd x_2^2+(\dd x_1+\dd x_2)^2.
\]
Along a sequential path this reduces to
\(\dd s=\sqrt2|\dd x_i|\) on every coordinate segment.  Hence every such
path from \(x\) to \(y\) has length at least
\[
  \sqrt2\left(
  |y_1-x_1|+|y_2-x_2|\right).
\]
Changing the first coordinate monotonically and then the second attains this
bound, proving the transition-distance formula in \textup{(D7)}.  The three
log-eigenvalues of \(D_3(y)\) are \(y_1,y_2,-y_1-y_2\), so its condition
target is exactly \(\mathcal H_k\).  Minimization of the transition distance
over that target proves the RDGC formula.

For \(x_0=(a,a)\), symmetry and convexity give a symmetric ambient AIRM
projection.  On the line \(y_1=y_2=s\), the target condition is
\(3|s|\le k\), so the projection is \(q=(k/3,k/3)\).  The same point minimizes
the \(\ell_1\) distance: for every \(y\in\mathcal H_k\),
\[
  2y_1+y_2\le k,
  \qquad
  y_1+2y_2\le k,
\]
and hence \(y_1+y_2\le2k/3\).  Since
\(|a-y_i|\ge a-y_i\),
\[
  \norm{x_0-y}_1
  \ge2a-y_1-y_2
  \ge2\left(a-\frac{k}{3}\right),
\]
with equality at \(q\).  Finally,
\[
  d_{\AI}(D_3(x_0),D_3(q))^2
  =2\left(a-\frac{k}{3}\right)^2
   +4\left(a-\frac{k}{3}\right)^2,
\]
which gives \textup{(D8)} and the strict ratio.
\end{proof}